\documentclass[11pt]{article}
\usepackage{}
\usepackage{enumerate}
\usepackage{mathbbold}
\usepackage{bbm}
\usepackage{amsfonts}
\usepackage[T1]{fontenc}
\usepackage{latexsym,amssymb,amsmath,amsfonts,amsthm}
\usepackage{graphicx, color}
\usepackage{subfigure}
\usepackage{epsf}
\usepackage{epstopdf}
\usepackage{amssymb}
\usepackage{authblk} 
\usepackage{amsthm}
\usepackage{multirow}
\usepackage{diagbox}
\usepackage{mathrsfs}
\usepackage[ruled]{algorithm2e}
\usepackage{hyperref}
\hypersetup{hypertex=true, 
colorlinks =true, 
linkcolor=black, 
anchorcolor=black,
citecolor=black}

\renewcommand{\theequation}{\arabic{section}.\arabic{equation}}

\newcommand{\R}{{\mathbb R}}

\newcommand{\N}{{\mathbb N}}
\newcommand{\C}{{\mathbb C}}

\newcommand{\be}{\begin{eqnarray}}
\newcommand{\ben}{\begin{eqnarray*}}
\newcommand{\en}{\end{eqnarray}}
\newcommand{\enn}{\end{eqnarray*}}
\newcommand{\ba}{\backslash}

\newcommand{\G}{\Gamma}

\newcommand{\Om}{\Omega}

\newtheorem{theorem}{Theorem}[section]
\newtheorem{lemma}[theorem]{Lemma}

\begin{document}
\title{\bf Identifying point sources for biharmonic wave equation from the scattered fields at sparse sensors}

\author[1]{Xiaodong Liu}
\author[2]{Qingxiang Shi}
\author[1,3]{Jing Wang} 
\makeatletter 
\renewcommand{\AB@affilnote}[1]{\footnotemark[#1]} 
\renewcommand{\AB@affillist}[1]{}
\renewcommand{\AB@affillist}[2]{} 
\makeatother

\date{}
\maketitle
\footnotetext[1]{State Key Laboratory of Mathematical Sciences, Academy of Mathematics and Systems Science,
Chinese Academy of Sciences, Beijing 100190, China. Email: xdliu@amt.ac.cn.} 
\footnotetext[2]{Yau Mathematical Sciences Center, Tsinghua University,
Beijing 100084, China. Email: sqxsqx142857@tsinghua.edu.cn.}
\footnotetext[3]{Corresponding author. School of Mathematical Sciences, University of Chinese Academy of Sciences, Beijing 100049, China, Email:wangjing23@amss.ac.cn}

\begin{abstract}
This work is dedicated to uniqueness and numerical algorithms for determining the point sources of the biharmonic wave equation using scattered fields at sparse sensors. We first show that the point sources in both $\R^2$ and $\R^3$ can be uniquely determined from the multifrequency sparse scattered fields. In particular, to deal with the challenges arising from the fundamental solution of the biharmonic wave equation in $\R^2$, we present an innovative approach that leverages the Fourier transform and Funk-Hecke formula. Such a  technique can also be applied for identifying the point sources of the Helmholtz equation. 
Moreover, we present the uniqueness results for identifying multiple point sources in $\R^3$ from the scattered fields at sparse sensors with finitely many frequencies. Based on the constructive uniqueness proofs, we propose three numerical algorithms for identifying the point sources by using multifrequency sparse
scattered fields. The numerical experiments are presented to verify the effectiveness and robustness of the algorithms.


\vspace{.2in}

\vspace{.2in} {\bf AMS subject classifications:}
35R30, 35P25, 78A46
\end{abstract}

\section{Introduction}
The inverse problems of  the biharmonic wave equation have many important applications such as offshore runway design, seismic cloaks, and platonic crystal \cite{2009application1, 2009application2, 2004application}. 
However, different from the extensive results for inverse second-order (e.g., acoustic and electromagnetic) wave scattering problems \cite{Amm3, Bao-Li-Book, 1998Inverse, 1990application, KirschGrinberg}, the research on the four-order biharmonic wave equations is still relatively limited. 
The scattered field for the biharmonic wave equation contains an exponential decay term, which brings great challenges for the analyses and numerical methods for the inverse problems.  We refer to some recent studies on the uniqueness and stability of the inverse source problems of the biharmonic wave equation \cite{Y.Guo, Y.Guo-Phaseless data, 2021application1, 2021application2, 2022application}. 
Generally speaking, the measurements are taken completely around the unknown sources. Unfortunately, this is unrealistic in many practical applications. What happens if the scattered fields are taken only at a finite number of sensors? At a fixed sensor, one may vary the frequency to obtain more data. This is still a small set of data, which indeed brings a lot of difficulties for the solvability of the inverse problems. 

This work is concerned with the uniqueness and numerical methods for identifying point sources for biharmonic wave equation from the multi-frequency scattered fields at sparse sensors. Different from the general source functions, the point sources are characterized by their locations and scattering strengths. This enables us to study the uniqueness analyses and numerical methods with less data. We show that the point sources can be identified, i.e., the number, locations and scattering strengths of the point sources can be uniquely recovered from the multi-frequency scattered fields at sparse sensors. In particular, we clarify the smallest number of sensors and frequencies to be used.

Note that the exponential decay part of the scattered field produces no contribution to the far field data. Thus, if the far field patterns are measured, the results in \cite{2020Identification} for Helmholtz equation can be directly extended to the case for the biharmonic wave equation. Based on the special structure of the scattered fields in $\R^3$, following the techniques proposed in \cite{JiLiu-nearfield}, we show the uniqueness of the point sources in $\R^3$ from the scattered fields for a finite number of sensors and all frequencies in an interval. However, difficulties arise due to the complex fundamental solution of the biharmonic wave equation in $\R^2$. With the help of the Funk-Hecke formula and the Fourier transform, we introduce a novel technique to prove the uniqueness result in $\R^2$. As a byproduct, such a novel technique can also be used to identify the point sources for Helmholtz equation in $\R^2$, which gives a positive answer to an open problem proposed in \cite{2020Identification}.
Considering the fact that for point sources we have finitely many parameters to be reconstructed, we also prove the uniqueness results by the scattered fields with finitely many properly chosen sensors and frequencies. 
An important feature is that, with the help of the exponential decay term, for identifying a single point source, we have no additional assumption on the smallest frequency.

Motivated by the constructive uniqueness proofs, we introduce three numerical algorithms for identifying the point sources. For a single point source in $\R^3$, we first offer a formula to compute the scattering strength using the scattered fields at three properly selected frequencies and any fixed sensor. We introduce an indicator for locating the position from the scattered fields with four properly chosen sensors but a single frequency. For multiple points in $\R^2$ and $\R^3$, we first apply the direct sampling method for locating all the positions from multi-frequency scattered fields at sparse sensors. Having located all the positions, for each point source, we give a formula for computing its scattering strength from multi-frequency scattered fields at a properly chosen sensor. All the indicators for locating the positions and the formulas for computing the scattering strengths involve only summation or integral operations and thus are simple, fast and stable.

The rest of this paper is arranged as follows. In Section \ref{Point sources fundamental}, we describe the mathematical model for point sources corresponding to biharmonic wave equation. In particular, we specify the representation of the scattered fields. Section \ref{Uniqueness results} is dedicated to uniqueness issues by giving the lower bound for the number of the measurement sensors. We begin with unique results in $\R^n$, where $n=2,3$, for identifying the point sources with frequencies in an interval. In the subsequent two subsections, we study the uniqueness from the scattered fields with finitely many frequencies. In Section \ref{Numerical methods}, following the uniqueness arguments, we introduce the numerical algorithms for locating the positions and computing the corresponding scattering strengths. Numerical simulations are presented in Section \ref{Numerical examples} to verify the validity and robustness of the proposed algorithms.

\section{Point sources in the biharmonic wave equation}
\label{Point sources fundamental}
We consider an array of $M$ point sources located at $z_1, z_2,\ldots, z_M\in \R^n$ in the homogeneous space $\R^n, n=2,3$, and denote by $\tau_1, \tau_2,\ldots, \tau_M \in \C \backslash \{0\}$ the corresponding scattering strengths. These point sources result in a scattered field $u^{s}$,solving the equation
\begin{align}\label{Main_eq}
\Delta^{2} u^s-k^4u^s=\sum_{m=1}^M\tau_m\delta_{z_m}\quad\mathrm{~in~}\mathbb{R}^n,
\end{align}
where $k>0$ is the wave number,  $\delta_{z_m}$ denoting the Dirac measure on $\mathbb{R}^n$ giving unit mass to the point $z_m, m=1, 2,\ldots,M$. To characterize a physical solution, we impose an analogue of the classical Sommerfeld radiation condition
\begin{align}\label{se2}
\partial_rw-ikw=o\left(r^{-\frac{n-1}{2}}\right),\ r=|x|\to\infty,\ w=u^s,\Delta u^s.
\end{align}
Specifically, the scattered field $u^s$ is given by
\begin{align}\label{Scattered field}
u^s(x,k)=\sum_{m=1}^M\tau_m\Phi_k(x,z_m),\quad x\in\mathbb{R}^n\backslash\{z_1,z_2,\ldots,z_M\}.
\end{align}
with
\begin{equation}
\Phi_k(x,y):= \left\{
         \begin{array}{ll} \frac{i}{8k^2}\left(H_0^{(1)}(k|x-y|)+\frac{2i}{\pi}K_0(k|x-y|)\right), & x,y \in \R^2, x\neq y, \\
               \frac{1}{8\pi k^2|x-y|}\left(e^{ik|x-y|}-e^{-k|x-y|}\right), & x,y \in \R^3, x\neq y,
                \end{array}
       \right.
   \label{equation_FS}
\end{equation}
where $H_0^{(1)}$ and $K_0$ are the Hankel function of first kind and the Macdonald’s function of order 0, respectively. Recall that
\begin{align}\label{FS1}
H_{0}^{(1)}(x):=J_{0}(x)+iY_{0}(x), \quad 
K_{0}(x)=2e^{-x}\sum_{j=0}^{\infty}\frac{(2j)!}{2^j(j!)^3}x^{j},
\end{align}
where $J_{0}(x)$ and $Y_{0}(x)$ are the first Bessel and Neumann functions of order 0 respectively.

For a finite number $L\in \N$, we denote by 
\ben
\Gamma_L:=\{x_1,x_2,\ldots,x_L\}\subset\mathbb{R}^n\backslash\{z_1,z_2,\ldots,z_M\}
\enn
the collection of the sensors for the scattered fields. The inverse problem is to reconstruct the number $M$ of the point sources,  the locations $z_m$ and the corresponding scattering strengths $\tau_m, m=1,2,\cdots M$ from the  multi-frequency sparse scattered fields 
\ben
u^{s}(x, k),\,\ x \in \Gamma_{L}, k\in K.
\enn
Here, the wave number set $K$ is an interval or a collection of finitely many frequencies.

\section{Uniqueness}
\label{Uniqueness results}
\setcounter{equation}{0}
As with many other inverse problems, our primary interest is the uniqueness issue. For identifying the finitely many point sources, we expect that the number $L$ of the measurement sensors is as small as possible. This section is dedicated to presenting the uniqueness results by giving such a lower bound for $L$.
Throughout the paper, we always use the following notations
\begin{equation}
r:=|x-z|,\quad r_m:=|x-z_m|, \quad r_{l,m}:=|x_l-z_m|,\qquad  m=1,2,\ldots,M,\ l=1,2,\ldots L, 
\nonumber
\end{equation}
where $x\in \Gamma_L,\ z\in \R^n\backslash\{z_1,z_2,\ldots,z_M\}$. 

\subsection{Identification of multiple point sources with frequencies in an interval}
\label{Multi-Numerical}
Noting that by analyticity the scattered field is completely determined for all positive frequencies by only knowing them in some interval. Therefore, in this subsection, we consider multiple frequencies in a bounded interval $K:=(k_{-}, k_{+})$.

Denote by $\mathcal{C}_{l,m}\subset\R^2$ the circle centered at $x_l$ with radius $r_{l,m}$ for $l=1,2,\ldots, L,\ m=1,2,\ldots, M$. 
For $z\in \R^2$, denote by $\mathcal{N}(z)$ the number of circles $\mathcal{C}_{l,m}$ passing through $z$.
\begin{lemma}
\label{L-proof}
Let $L\geq 2M+1$, and assume that any three points in $\G_L$ are not collinear. Then, 
\begin{equation}
\mathcal{N}(z)
\begin{cases}
=L, & z\in\{z_1,z_2,\ldots,z_M\}; \\
 \\
\leq2M, & otherwise.
\end{cases}
\nonumber
\end{equation}
Furthermore, for any $z_{m^*}\in \{z_1,z_2,\ldots,z_M\}$, we have at least $L-2(M-1)$ sensors, without generality, denoted by $x_1,x_2,\ldots,x_{L-2(M-1)}\in\G_L$, such that
\begin{equation}
r_{m^*}\neq r_m,\quad \forall x\in\{x_1,x_2,\ldots,x_{L-2(M-1)}\}, m\neq m^*.
\label{necessary assumption}
\end{equation}
\end{lemma}
\begin{proof}
The equality $\mathcal{N}(z_m)
=L$ is clear since  $z_m\in\mathcal{C}_{l,m}$ for all $l=1,2,\ldots, L$.
For $z\in \R^2\backslash\{z_1, z_2,\ldots, z_M\}$, the circles $\mathcal{C}_{l,m}$ passing through $z$ must pass through some $z_m$. That is, $z,z_m\in \mathcal{C}_{l,m}$ and the sensor $x_l\in \Pi:= \{ y\in \R^2 | (z-z_m)\cdot  [y-(z+z_m)/2]=0 \}$.
By the assumption that any three sensors are not collinear, there are at most two sensors on $\Pi$. That is to say there are at most two circles passing through $z$ and $z_m$ simultaneously. Therefore, $\mathcal{N}(z)\leq2M,\ z\in\R^2\backslash\{z_1, z_2,\ldots, z_M\}$.

In addition, for all $m\neq m^*$, $r_{m^*}=r_m$ holds for at most two sensors because any three sensors in $\Gamma_L$ are not collinear. Then, for fixed $z_{m^*}$, there are at most $2(M-1)$ sensors such that $r_{m^*}=r_m$. In other words, we have at least $L-2(M-1)$ senors for which \eqref{necessary assumption} holds.
\end{proof}
We use $\mathcal{F}$ and $\mathcal{F}^{-1}$ to denote the Fourier transform and the inverse Fourier transform:
\begin{equation}
\mathcal{F}[\varphi](y):=\frac{1}{(2\pi)^{n/2}}\int_{\R^n}\varphi(x)e^{-ix\cdot y}dx,\quad 
\mathcal{F}^{-1}[\psi](y):=\frac{1}{(2\pi)^{n/2}}\int_{\R^n}\psi(x)e^{ix\cdot y}dx,\quad\varphi, \psi\in \mathscr{S}(\R^{n}),
\nonumber
\end{equation}
where $\mathscr{S}(\R^{n})$ is the Schwartz space. Note that
$\mathcal{F}[\delta]=\mathcal{F}^{-1}[\delta]={1}/{2\pi}. $
Furthermore, we define the generalized function $\hat{\delta}\in \mathscr{S}^\prime(\R^{2})$ by
\ben
\langle \hat{\delta}(r-r_m),\varphi(z)\rangle:=\frac{1}{r_m}\int_{r=r_{m}}\varphi(z) ds,\quad \varphi \in \mathscr{S}(\R^{2}).
\enn
\begin{theorem}\label{Multipoint sources-11}
For $M$ point sources in $\R^2$, let $L\geq 4M-1$ and assume that any three sensors in $\Gamma_L$ are not collinear. Then, the number $M$, the positions $z_1,z_2,\ldots,z_M$ and the scattering strengths $\tau_1, \tau_2,\ldots, \tau_M$  of the point sources are uniquely determined by the scattered fields $u^s(x,k)$ for finitely many $x\in \Gamma_L$ and all $k\in (k_{-}, k_{+})$.
\end{theorem}
\begin{proof}
To show the uniqueness of the number $M$ and the positions  $z_1,z_2,\ldots,z_M$, we first study a simpler case with $\tau_m\in\R, m=1,2,\ldots,M$ and sketch the proof for the general case with  $\tau_m\in\C, m=1,2,\ldots,M$. 
\begin{itemize}
    \item {\bf Case one:} $\tau_m \in\R$, $m=1,2,\ldots,M$.  
\end{itemize}    
Recall the  Funk-Hecke formula \cite{1998Inverse, 2017Funk-Hecke formula},
\ben
J_0(|y|)=\frac{1}{2\pi}\int_{0}^{2\pi}e^{i|y|\cos(\theta-\phi)}d\theta,\quad y=|y|(\cos \phi, \sin \phi)\in \R^2.
\enn
We have 
\begin{equation}\label{Fourier-beseel transformation}
\begin{aligned}
\int_{0}^{+\infty}kJ_{0}(kr_{m})J_{0}(kr)dk 
 &=\frac{1}{2\pi}\int_{0}^{+\infty}\int_{0}^{2\pi}J_{0}(kr_m)e^{ikr\cos(\theta-\phi)}kd\theta dk \cr
 &=\frac{1}{2\pi}\int_{\R^{2}}J_{0}(|y|r_m)e^{iy\cdot (x-z)}dy \cr
 &=\mathcal{F}^{-1}[J_{0}(|y|r_m)](x-z), \quad x\in\G_L, z\in\R^2.
 \end{aligned}
\end{equation}
Straightforward calculations show that,
\begin{equation}
\begin{aligned}
\langle\mathcal{F}^{-1}\left[J_0(|y|r_m)\right],\mathcal{F}\left[\varphi(y)\right]\rangle
=\langle J_0(|y|r_m),\varphi(y)\rangle
=\int_{\R^2}\varphi(y)J_0(|y|r_m)dy,\quad \forall \varphi \in \mathscr{S}(\R^{2}),
\end{aligned}
\nonumber
\end{equation}
and
\begin{equation}
\begin{aligned}
\langle\hat{\delta}(r-r_m),\mathcal{F}\left[\varphi(y)\right]\rangle
&=\frac{1}{2\pi r_m}\int_{r=r_m}\int_{\R^2}\varphi(y)e^{-iy\cdot (x-z)}dyds\\
&=\frac{1}{2\pi r_m}\int_{\R^2}\int_{0}^{2\pi}\varphi(y)e^{i|y|r_m \cos(\theta-\phi-\pi)}r_md\theta dy\\
&=\int_{\R^2}\varphi(y)J_0(|y|r_m)dy,\quad \forall \varphi \in \mathscr{S}(\R^{2}).
\end{aligned}
\nonumber
\end{equation}
Therefore, we deduce that
\begin{equation}
    \mathcal{F}^{-1}\left[J_0(|y|r_m)\right](x-z)=\hat{\delta}(r-r_m),\quad \ x\in\G_L, z\in\R^2.  
    \label{generalized function}
\end{equation}
From \eqref{Scattered field} and \eqref{FS1}, noting that the functions $J_0, Y_0$ and $K_0$ are real-valued, we have 
\begin{equation}
\Im (u^s(x,k))=\frac{1}{8k^2}\sum_{m=1}^M\tau_mJ_{0}(kr_m),\quad  x\in \Gamma_L, k\in (k_{-}, k_{+}).
\nonumber
\end{equation}
Furthermore, using \eqref{Fourier-beseel transformation} and \eqref{generalized function}, we derive that
\begin{equation}\label{indicator1} 
\begin{aligned}
 I_{\mathbb{R}^2}^{(1)}(r):=&\int_{0}^{+\infty} \Im( 8 k^{2} u^{s}(x, k)) k J_{0}(k r) d k \\
=&\int_{0}^{+\infty} \sum_{m=1}^{M} \tau_{m} k J_{0}\left(k r_{m}\right) J_{0}(k r) d k \\
=&\sum_{m=1}^{M} \tau_{m} \mathcal{F}^{-1}[J_{0}(k r_m)](x-z) \\
=&\sum_{m=1}^{M} \tau_{m} \hat{\delta}(r-r_{m}),\quad x\in\G_L, z\in\R^2.  
\end{aligned}
\end{equation}
This implies $I_{\R^2}^{(1)}(r)$ vanishes almost everywhere and must blow up at $r=r_m$ when $r_m\neq r_{m^*}$, $\forall m\neq m^*$.
By Lemma \ref{L-proof}, we know that there are at least $L-2(M-1)\geq 2M+1$ measurement points such that, for any $ z_{m^*}\in \{z_1,\,z_2,\cdots,\,z_M\}$,
$r_{l,m^*}\neq r_{l,m},\ \forall m\in \{1,2,\cdots, M\}\backslash\{m^*\}$.
Therefore, we can find at least $2M+1$ circles $\mathcal{C}_{l,m^*}$ centered at $x_l$ passing through $z_{m^*}$, $l=1,2,\ldots,2M+1$. By applying Lemma \ref{L-proof} again, the number $\mathcal{N}(z)$ can distinguish the point sources $\{z_1,z_2,\cdots, z_M\}$ in $\R^2$. 

\begin{itemize}
    \item {\bf Case two:} $\tau_m\in C, m=1,2,...,M$. 
\end{itemize}    
    This case can be proved by replacing \eqref{indicator1} with
\begin{equation}\label{indicator2}
\begin{aligned}
I_{\mathbb{R}^2}^{(2)}(r):&=\int_{0}^{+\infty}-8 k^{2} i u^{s}(x, k) k J_{0}(k r) d k \\
&=\sum_{m=1}^{M} \tau_{m}\int_{0}^{+\infty} \left(H_{0}^{(1)}\left(k r_{m}\right)+\frac{2 i}{\pi} K_{0}\left(k r_{m}\right)\right) k J_{0}(kr) d k\\
&=\sum_{m=1}^{M} \tau_m\mathcal{F}^{-1}\left[H_{0}^{(1)}\left(|y| r_{m}\right)+\frac{2 i}{\pi } K_{0}\left(|y| r_{m}\right)\right]\\
&=\sum_{m=1}^{M}\left(\frac{8r_{m}^2}{i}\right)\frac{\tau_m}{r^{4}-r_{m}^{4}}\mathcal{F}^{-1}\bigg[(\Delta_{y}^{2}-r_{m}^{4})\Phi_{r_m}(y,0)\bigg]\\
&=\sum_{m=1}^{M}\left(\frac{8r_{m}^2}{i}\right) \frac{\tau_m}{r^{4}-r_{m}^{4}}\mathcal{F}^{-1}[\delta]\\
&=\sum_{m=1}^{M}\left(\frac{4r_{m}^2}{\pi i}\right)\frac{\tau_m}{r^{4}-r_{m}^{4}},\quad x\in\G_L, z\in\R^2.
\end{aligned} 
\end{equation}

Having determined the locations, we proceed to demonstrate that the corresponding scattering strengths $\{\tau_1,\tau_2,\ldots,\tau_M\}$ can also be uniquely recovered. Actually, for any $z_{m^{*}}\in \{z_1,z_2,\ldots,z_M\}$, we can always take some $x\in \Gamma_L$, such that $r_{m^{*}}\neq r_m$ when $m\neq m^*$. Similar to formula \eqref{indicator2}, we obtain
\begin{equation}
\begin{aligned}
I_{\R^2}^{(3)}(r):&=\int_{0}^{+\infty} \left(H_{0}^{(1)}\left(k r_{m^*}\right)+\frac{2 i}{\pi} K_{0}\left(k r_{m^*}\right)\right) k J_{0}(kr) d k=\left(\frac{4r_{m}^2}{\pi i}\right)\frac{1}{r^{4}-r_{m^*}^{4}}.
\nonumber
\end{aligned}
\end{equation}
Therefore,
\begin{equation}
\frac{I_{\R^2}^{(2)}(r_{m^*})}{I_{\R^2}^{(3)}(r_{m^*})}=\tau_{m^{*}}+\sum_{m\neq m^*}\frac{\tau_m}{r_{m^*}^{4}-r_{m}^{4}}\cdot 0=\tau_{m^{*}}.
\label{tau-R2}
\end{equation}
The proof is complete.
\end{proof}

It is worth noting that Theorem \ref{Multipoint sources-11} can be directly generalized to the Helmhotz equation. This actually gives a positive answer to an open problem proposed in \cite{2020Identification}. 
The following theorem shows the corresponding uniqueness result in $\R^3$.

\begin{theorem}\label{Multipoint sources-12}
For $M$ point sources in $\mathbb R^3$, let $L\geq 6M-2$ and assume that any four sensors in $\G_L$ are not coplanar. Then, the number $M$, the positions $z_1,z_2,\ldots,z_M$ and the scattering strengths $\tau_1, \tau_2,\ldots, \tau_M$  of the point sources are uniquely determined by the scattered fields $u^s(x,k)$ for finitely many $x\in \Gamma_L$ and all $k\in (k_{-}, k_{+})$.
\end{theorem}
\begin{proof}
From \eqref{Scattered field} and \eqref{equation_FS} in $\R^3$, we have the scattered field
\begin{equation}
u^s(x,k)=\sum_{m=1}^M\frac{\tau_m}{8\pi k^2r_m}\left(e^{ikr_m}-e^{-kr_m}\right),\quad x\in\G_L, k\in\R^{+}.
\label{R3-scatter}
\end{equation}
We consider the following integral representation:
\begin{equation}
\begin{aligned}
I_{\mathbb{R}^3}(r, k_{+}):&=\int_{0}^{k_{+}}8 k^{2}u^{s}(x,k)e^{-ikr}dk \\
&=\frac{1}{\pi}\int_{0}^{k_{+}}\sum_{m=1}^{M}\frac{\tau_{m}}{r_{m}}\left[e^{ik(r_m-r)}-e^{-k(r_m+ir)}\right]dk\\
&=\sum_{m=1}^{M}\frac{\tau_{m}}{r_{m}}\left[\frac{1}{\pi}\int_{0}^{k_{+}}\cos (k(r_m-r))dk+\frac{i}{\pi}\int_{0}^{k_{+}}\sin (k(r_m-r))dk-\frac{1-e^{-k_+(r_m+ir)}}{\pi (r_m+ir)}\right].
\end{aligned}
\label{indicator3}
\end{equation}
Noting that 
\begin{equation}
\lim\limits_{k_+\to+\infty}\frac{1}{\pi}\int_{0}^{k_{+}}\cos (k(r_m-r))dk=\delta\left(r-r_{m}\right),
\nonumber
\end{equation}
and $\frac{i}{\pi}\int_{0}^{k_{+}}\sin (k(r_m-r))dk$ and $\frac{1-e^{-k_+(r_m+ir)}}{\pi (r_m+ir)}$ are bounded, $\forall k_{+}>0$, therefore the function $I_{\R^3}(r, k_{+})$ is bounded almost everywhere and must blow up at $r=r_m$ as $k_{+}\rightarrow +\infty$, when $r_m\neq r_{m^*}$, $\forall m\neq m^*$.
The uniqueness of $\{z_1,z_2,\ldots,z_{M^*}\}$ follows by using Lemmas 3.2 and 3.3 in \cite{JiLiu-nearfield} under the assumptions $L\geq6M-2$ and that any four sensors in $\Gamma_L$ are not coplanar.


Having determined the locations, we show in the following that the corresponding scattering strengths $\tau_m,m=1,2,\ldots,M,$ can also be uniquely determined. For any $z_{m^{*}}\in \{z_1,z_2,\ldots,z_M\}$, there exists some $x\in \Gamma_L$ such that $r_{m^{*}}\neq r_m$, for $m\neq m^*$. Then, we have the asymptotic expression for $k_{+}$:
\begin{equation}
\begin{aligned}
 \int_{0}^{k_+}8\pi k^{2}u^{s}\left(x,k\right)e^{-ikr_{m^*}}dk 
 =\int_{0}^{k_+}\sum_{m=1}^{M}\frac{\tau_{m}}{r_{m}}\left(e^{ikr_{m}}-e^{-kr_{m}}\right)e^{-ikr_{m^{*}}}dk 
  =\frac{\tau_{m^{*}}k_+}{r_{m^{*}}}+O(1)
\nonumber
\end{aligned}
\end{equation}
as $k_+\to +\infty$, therefore, we obtain the formula for calculating $\tau_{m^*}$:
\begin{equation}
\tau_{m^{*}}=\lim_{k_+\to+\infty}\frac{r_{m^{*}}}{k_{+}}\int_{0}^{k+}8\pi k^{2}u^{s}\left(x,k\right)e^{-ikr_{m^{*}}}dk.
\label{tau-R3}
\end{equation}
The proof is complete.
\end{proof}
\subsection{Identification of a single point source with three frequencies}\label{Single-Numerical}
Considering the fact that for point sources we have only finitely many parameters to be reconstructed, we also expect the uniqueness results from the scattered fields with not only finitely many sensors but also finitely many frequencies.

We begin with the simplest situation with a single point source located at $z_1\in\R^3$ with scattering strength $\tau_1\in\C$. Clearly, given $z_1$, we have $
\tau_{1}={u^{s}\left(x, k\right)}/{\Phi_{k}\left(x, z_{1}\right)}$,
i.e., $\tau_1$ can be determined by the scattered field $u^s(x,k)$ with fixed $x \in\mathbb{R}^3\backslash\{z_1\}$ and $k>0$. Difficulty arises if we want to recover the location $z_1$ since the scattered field depends nonlinearly on $z_1$. The following theorem shows that, with given $\tau_1$, the location $z_1$ can be determined by the scattered fields at four sensors.
\begin{theorem}\label{Single point-1}
Consider single point source in $\R^3$. For a fixed frequency $k>0$, given $\tau_1$, the distance $\left|x-z_1\right|$ is uniquely determined by $u^{s}(x, k)$ at a single sensor $x \in\mathbb{R}^3\backslash\{z_1\}$. Furthermore, the location $z_1$ can be uniquely determined by four sensors in $x_1,x_2,x_3,x_4 \in\mathbb{R}^3\backslash\{z_1\}$ that are non-coplanar.
\end{theorem}
\begin{proof}
With the help of \eqref{Scattered field} and \eqref{equation_FS}, denote by $y:=k|x-z_{1}|>0$, we have
\begin{equation}\label{111}
f(y):=64k^2\pi^2\left|\frac{u^{s}(x,k)}{\tau_1}\right|^2
=\left|\frac{e^{iy}-e^{-y}}{y}\right|^2
=\frac{1-2\cos(y)e^{-y}+e^{-2y}}{y^{2}}, \quad y>0.
\end{equation}
We claim that the function $f$ is injective, and therefore, given $|\tau_1|$, the distance $\left|x-z_1\right|$ is uniquely determined by the phaseless scattered field $|u^s(x,k)|$ at a fixed sensor $x$.
To do so, we have
\begin{equation}
f^{\prime}(y)=\frac{2e^{-y}[\sqrt2y\sin\left(y+{\pi}/{4}\right)+2\cos y-(y+1)e^{-y}-e^{y}]}{y^{3}},\quad y>0
\label{222}
\end{equation}
and we will show that $f^{\prime}(y)<0$, for $y>0$, i.e.,
\begin{equation}
g(y):=\sqrt2y\sin\left(y+{\pi}/{4}\right)+2\cos y-(y+1)e^{-y}-e^{y}<0,\quad y>0.
\label{333}
\end{equation}

The proof of \eqref{333} is divided into three cases.
\begin{itemize}
    \item {\bf Case one:} $y\in A:=\bigcup\limits_{m\in\N}\left(\frac{\pi}{2}+2m\pi,\frac{3\pi}{2}+2m\pi\right]$.
\end{itemize}

From the obvious inequalities
\begin{equation}
 \sin\left(y+\frac{\pi}{4}\right)\leq\frac{\sqrt{2}}{2},\ \cos y<0,\ -(y+1)e^{-y}<0,\quad y\in A,
\nonumber
\end{equation}
we can estimate
\begin{equation}
\begin{aligned}
g(y)<y-e^{y}=:h_1(y),\quad y\in A.
\label{case1-estimate1}
\end{aligned}
\end{equation}
Noting that $h_1^{\prime}(y)=1-e^{y}\leq0$, $y\geq0$, i.e., $h_1(y)$ is a monotonically decreasing function for $y>0$, and thus 
\be\label{case1-estimate2}
h_1(y)\leq h_1(0)=-1,\quad y>0.
\en
Therefore, the inequality \eqref{333} for $y\in A$ follows from \eqref{case1-estimate1}-\eqref{case1-estimate2}.

\begin{itemize}
   \item {\bf Case two:} $y\in B:=\bigcup\limits_{m\in\N}\left(\frac{3\pi}{2}+2m\pi,2\pi+2m\pi\right]$.
\end{itemize}

In the second case, based on the facts that
\begin{equation}
 \sin\left(y+\frac{\pi}{4}\right)\leq\frac{\sqrt{2}}{2},\ \cos y\leq 1,\quad y\in B,
\nonumber
\end{equation}
we can obtain
\begin{equation}
\begin{aligned}
g(y)\leq y+2-(y+1)e^{-y}-e^{y}=:h_2(y),\quad y\in B.
\end{aligned}
\label{case2-estimate1}
\end{equation}
Straightforward calculations show that
\begin{align*}
&h_2^{\prime}(y)=1-e^{y}+ye^{-y},\quad y\geq0,\cr
&h_2^{\prime\prime}(y)=e^{-y}-e^{y}-ye^{-y}<e^{-y}-e^{y}<0,\quad y>0,
\end{align*}
with this, we find that $h_2^{\prime}(y)<h_2^{\prime}(0)=0$ for $y>0$ and furthermore
\be\label{case2-estimate2}
h_2(y)<h_2(0)=0,\quad y>0.
\en
The inequality \eqref{333} for $y\in B$ now follows from \eqref{case2-estimate1}-\eqref{case2-estimate2}.

\begin{itemize}
   \item {\bf Case three:} $y\in C:=\bigcup\limits_{m\in\N} (2m\pi,\frac{\pi}{2}+2m\pi]$.
\end{itemize}

In the third case,  we have 
\begin{equation}
g^{\prime}(y)=\sqrt{2}\sin\left(y+\frac{\pi}{4}\right)+\sqrt{2}y\cos\left(y+\frac{\pi}{4}\right)-2\sin y+ye^{-y}-e^{y},\quad y\in C,
\label{g'}
\end{equation}
 and
\begin{equation}
g^{\prime\prime}(y)=g^{\prime\prime}_1(y)+g^{\prime\prime}_2(y)+g^{\prime\prime}_3(y),\quad y\in C,
\nonumber
\end{equation}
 where
\begin{equation}
\begin{aligned}
g^{\prime\prime}_1(y)&:=2\sqrt{2}\cos\left(y+\frac{\pi}{4}\right),\quad y\in C,\\
g^{\prime\prime}_2(y)&:=e^{-y}-ye^{-y},\quad y\in C,\\
g^{\prime\prime}_3(y)&:=-\sqrt{2}y\sin\left(y+\frac{\pi}{4}\right)-2\cos y-e^{y},\quad y\in C.
\end{aligned}
\nonumber
\end{equation}
Noting that
\begin{equation}
\cos\left(y+\frac{\pi}{4}\right)<\frac{\sqrt{2}}{2},\quad y\in C,
\nonumber
\end{equation}
we obtain
\begin{equation}
g^{\prime\prime}_1(y)<2,\quad y\in C.
\label{g''-1}
\end{equation}
From the fact $e^{-y}<1$ for $y>0$, we get
\begin{equation}
g^{\prime\prime}_2(y)<e^{-y}<1,\quad y\in C.
\label{g''-2}
\end{equation}
According to the fact that
\begin{equation}
 \sin\left(y+\frac{\pi}{4}\right)\geq\frac{\sqrt{2}}{2},\quad y\in C.
\nonumber
\end{equation}
we derive
\begin{equation}
g^{\prime\prime}_3(y)\leq -y-2\cos y-e^{y}=:h_3(y), \quad y\in C.
\nonumber
\end{equation}
From $h_3^{\prime}(y)=-1+2\sin y-e^{y}\leq 1-e^y\leq 0$ for $y\geq0,$
we deduce that $h_3(y)\leq h_3(0)=-3$ for $y>0$, and therefore
\begin{equation}
g^{\prime\prime}_3(y)\leq h_3(y)\leq -3, \quad y\in C.
\label{g''-3}
\end{equation}
Combining \eqref{g''-1}, \eqref{g''-2} and \eqref{g''-3}, we have 
\be\label{g''<0}
g^{\prime\prime}(y)<0, \quad y\in C.
\en
Noting that the set $C$ is composed of disjoint intervals, to prove $g^{\prime}(y)<0$, $y\in C$, it suffices to show that $g^{\prime}(2m\pi)\leq0$, $m\in \N$. Obviously, $g^{\prime}(0)=0$. Next, we prove $g^{\prime}(2m\pi)\leq0$, $m\in \N\backslash\{0\}$. From equation \eqref{g'}, we have $g^{\prime}(2m\pi)=1+2m\pi+2m\pi e^{-2m\pi}-e^{2m\pi}$, $m\in \N$. We define 
\begin{equation}
I(z):=1+z+ze^{-z}-e^{z},\quad z\in \left[2\pi, \infty\right), 
\nonumber
\end{equation}
then 
\begin{equation}
\begin{aligned}
I^{\prime}(z)=1+(1-z)e^{-z}-e^{z}<2-e^{2\pi}<0,\quad z\in [2\pi, \infty). 
\end{aligned}
\nonumber
\end{equation}
Therefore, 
\ben
I(z)<I(2\pi)=1+2\pi+2\pi e^{-2\pi}-e^{2\pi}<0,\quad z\in [2\pi, \infty). 
\enn
This implies $g^{\prime}(2m\pi)=I(2m\pi)<0, \quad m\in \N \backslash\{0\}$.
Consequently, from \eqref{g''<0} we now have
\begin{equation}
g^{\prime}(y)<g^{\prime}(2m\pi)\leq 0,\quad y\in \left(2m\pi, \frac{\pi}{2}+2m\pi\right],\ m\in \N. 
\label{g'<0}
\end{equation}
Furthermore, from {\bf Case two} we have $g(y)<0$, $y\in B$, and in particular  $g(2m\pi)<0$, $m\in \N \backslash \{0\}$. Then, $g(2m\pi)\leq 0$, $m\in \N$ based on the obvious fact that $g(0)=0$. Therefore, by \eqref{g'<0}, we have 
\begin{equation}
g(y)<g(2m\pi)\leq 0, \quad y\in \left(2m\pi, \frac{\pi}{2}+2m\pi\right],\ m\in \N. 
\nonumber
\end{equation}
That is to say $g(y)<0$, $y\in C.$

Up to now, we have proved that $\left|x-z_1\right|$ is uniquely determined by $\left|u^{s}(x, k)\right|$ at a single sensor $x \in\mathbb{R}^3\backslash\{z_1\}$ with a fixed frequency $k>0$. Under the condition that the four sensors $x_1, x_2, x_3, x_4\in \R^3\backslash \{z_1\}$ are non-coplanar, the uniqueness of the position  $z_1$ now follows from Theorem 3.1 of \cite{2020Identification}. The proof is complete.
\end{proof}

The following theorem gives a uniqueness result without a priori information about the scattering strength.
\begin{theorem}\label{Single point-2}
Consider a single point source in $\R^3$. Both the location $z_1$ and the strength $\tau_1$ can be uniquely determined by the scattered fields
\begin{equation}
u^{s}(x,k),\quad x\in \G_4:=\{x_1,x_2,x_3,x_4\},\ k\in K:=\{k_0, 2k_0,4k_0\},
\label{R3scattered field data}
\end{equation}
for some $k_0>0$ and four non-coplanar sensors $x_1,x_2,x_3,x_4$.
\end{theorem}
\begin{proof}
By \eqref{Scattered field} and \eqref{equation_FS}, we have
\begin{equation}
u^{s}(x,k)=\frac{\tau_1}{8\pi k^2|x-z_1|}\left(e^{ik|x-z_1|}-e^{-k|x-z_1|}\right),\quad x\in \G_4,\, k\in K.
\nonumber
\end{equation}
Straightforward calculations show that
\ben
&&\frac{4u^{s}(x,2k_{0})}{u^{s}(x,k_{0})}=\cos(k_{0}|x-z_1|)+e^{-k_0|x-z_1|}+i\sin(k_{0}|x-z_1|),\\
&&\frac{4u^{s}(x,4k_{0})}{u^{s}(x,2k_{0})}=\cos(2k_{0}|x-z_1|)+e^{-2k_0|x-z_1|}+i\sin(2k_{0}|x-z_1|).
\enn
Furthermore, we deduce from the above two equalities that 
\begin{equation}
\Re\left(\frac{4u^{s}(x,2k_{0})}{u^{s}(x,k_{0})}\right)-\frac{\Im\left(\frac{u^{s}(x,4k_{0})}{u^{s}(x,2k_{0})}\right)}{2\Im\left(\frac{u^{s}(x,2k_{0})}{u^{s}(x,k_{0})}\right)}=e^{-k_0|x-z_1|},\quad x\in\G_4.
\label{444}
\end{equation}
The left hand side of \eqref{444} is given by the scattered fields while the right hand side of \eqref{444} is a strictly monotonically decreasing function with respect to the distance $|x-z_1|$. Therefore, the distance $|x-z_1|$ can be uniquely determined by sparse scattered field data \eqref{R3scattered field data}. Similarly,  again using Theorem 3.1 in \cite{2020Identification}, we deduce that the location $z_1$ can be uniquely determined by four sensors $x_1,x_2,x_3,x_4$, which are non-coplanar. Furthermore, with the help of the representation of the scattered fields, the strength $\tau_1$ can be obtained by the formula
\begin{equation}
\tau_1=\frac{8\pi k_0^2u^{s}(x,k_0)|x-z_1|}{e^{ik_0|x-z_1|}-e^{-k_0|x-z_1|}}.
\label{555}
\end{equation}
for any $x\in\G_4$. The proof is complete.
\end{proof}

We want to emphasize that there is no assumption on $k_0$ in the above theorem. The corresponding uniqueness of Theorem \ref{Single point-2} for the Hemholtz equation is still open.


\subsection{Identification of multiple point sources with finitely many frequencies}\label{Uniqueness inR3-2}
We consider the identification of $M$ point sources in $\R^3$ from finitely many scattered fields 
\ben
u^s(x,k_j),\quad x\in\G_L,\, k_{j}=j k_{0},\, j=1, 2, \ldots, J
\enn
for $J>4M$ and some $k_0$ satisfying
\be\label{k0-assumption}
0<k_0\leq\frac{\pi}{2R}.
\en
Here, $R>0$ denotes the radius of the smallest ball $B_{R}$ containing all positions $z_m$ and all sensors $x\in \Gamma_L$. 
Note that, at a fixed sensor $x\in \Gamma_L$, two different positions $z_{m_1}\neq z_{m_2} $ may yield the same distances $r_{m_{1}}, r_{m_{2}}$. Consequently, from the representation \eqref{R3-scatter}, we observe that two terms of $u^s(x,k)$ may cancel each other if $\tau_{m_1}=-\tau_{m_2}$. To address this issue, following \cite{GriesmaierSchmiedecke, 2020Identification}, we define the set
\begin{equation}
\mathcal M_x:=\text{supp}\left(\sum_{m=1}^{M}\tau_m\delta(r-r_{m})\right).
\nonumber
\end{equation}
Denote by  $M^*$ be the cardinality of $\mathcal M_x$. Clearly, $M^*\leq M$. We then rewrite the scattered fields as
\begin{equation}
8\pi k^2u^s(x,k_j)=\sum_{m=1}^{M^*}\tau_m^*\xi_m^j+\sum_{m=1}^{M^*}(-\tau_m^*)\eta_m^j,\quad x\in\G_L, j=1,2,\ldots,J,
\nonumber
\end{equation}
where 
\ben
\xi_{m}=e^{ik_0f_m}\quad\mbox{and}\quad
\eta_{m}=e^{-k_0f_m},\quad f_m\in \mathcal M_x,\, m=1,2,\ldots,M^*.
\enn
Obviously,
 \begin{equation}\label{xi1-eta<1}
    |\xi_m|=1,\quad |\eta_m|<1,\quad m=1,2,\ldots,M^*.
\end{equation}
Furthermore, we have $\xi_i\neq \xi_j$ for $i\neq j$ by the assumption \eqref{k0-assumption} and $\eta_i\neq \eta_j$ for $i\neq j$ by the fact that  $F(t):=e^{-k_0t}, t\in\R$, is a strictly monotonic  function.
Finally, we introduce the diagonal matrix
\begin{equation}
D:=\text{diag}(\tau_{1}^{*}\xi_1,\tau_{2}^{*}\xi_2,\ldots,\tau_{M^*}^{*}\xi_{M^*},-\tau_{1}^{*}\eta_1,-\tau_{2}^{*}\eta_2,\ldots,-\tau_{M^*}^{*}\eta_{M^*})\in \C^{2M^{*}\times 2M^{*}},
\nonumber
\end{equation}
the Vandermonde matrices
\begin{equation}
V_i:=(v_{\xi_1},v_{\xi_2},\ldots,v_{\xi_{M^{*}}},v_{\eta_1},v_{\eta_2},\ldots,v_{\eta_{M^{*}}})\in \C^{i\times 2M^{*}},\quad i\in \mathbb{Z}^{+},
\nonumber
\end{equation}
where 
\begin{equation}
v_{\xi_m}=(1, \xi_m,\xi_m^2,\ldots,\xi_m^{i-1})^T\quad\mbox{and}\quad
v_{\eta_m}=(1, \eta_m,\eta_m^2,\ldots,\eta_m^{i-1})^T,\quad m=1,2,\ldots,M^*
\nonumber
\end{equation}
and the data matrix $U(x)\in \C^{(J-2M)\times(2M+1)}$, $x\in\Gamma_L$ by
\begin{align*}
U(x):=\left[\begin{matrix}
  8\pi k_1^2u^s(x,k_1)&8\pi k_2^2u^s(x,k_2)&\cdots&8\pi k_{2M+1}^2u^s(x,k_{2M+1})\\\\
  8\pi k_2^2u^s(x,k_2)&8\pi k_3^2u^s(x,k_3)&\cdots&8\pi k_{2M+2}^2u^s(x,k_{2M+2})\\\\
  \vdots&\vdots&\ddots&\vdots\\
  8\pi k_{J-2M}^2u^s(x,k_{J-2M})&8\pi k_{J-2M+1}^2u^s(x,k_{J-2M+1})&\cdots&8\pi k_J^2u^s(x,k_J)\\\\
\end{matrix}\right].
\end{align*}

\begin{theorem}\label{Uniqueness-MultiPoint-R3}
Consider $M$ point sources in $\R^3$. Let $L\geq6M-2$, and assume that any four sensors in $\Gamma_L$ are non-coplanar, then the positions $z_1,z_2,\ldots,z_M$ can be uniquely determined by the data matrix $U(x)$ for all $x\in \Gamma_L$. Having located all the positions $\{z_1,z_2,\ldots,z_{M}\}$, we can take a sensor $x\in\Gamma_L$ such that
\begin{equation}\label{r1neqr2}
r_{m_{1}}\neq r_{m_{2}}, \quad if \quad m_1\neq m_2
\end{equation}
and the corresponding $\{\tau_1,\tau_2,\ldots,\tau_{M}\}$ are uniquely determined by the scattered fields $u^s(x,k_j)$, $j=1,2,\ldots,J.$
\end{theorem}
\begin{proof}
We prove the theorem in the following four parts:
\begin{itemize}
    \item \textbf{Part one: We show that the data matrix $U$ has a factorization} 
    \begin{equation}\label{U-factorization}
        U=V_{J-2M} D V_{2M+1}^T
    \end{equation}
    {\bf and} 
    \begin{equation}\label{RangeIdentity}
        \mathcal{R}(U)=\mathcal{R}(V_{J-2M}).
    \end{equation}
    The factorization \eqref{U-factorization} follows from a straightforward calculation. This further implies that $\mathcal{R}(U)\subset\mathcal{R}(V_{J-2M})$. Thus \eqref{RangeIdentity} follows by showing $\mathcal{R}(V_{J-2M})\subset \mathcal{R}(U)$. To do so, for any $\phi\in \mathcal{R}(V_{J-2M})$, let $\psi\in \C^{2M^*\times 1}$ be such that $\phi=V_{J-2M}\psi$ and we show that $\phi\in \mathcal{R}(U)$. Note that $\xi_1,\xi_2,\ldots,\xi_{M^*},\, \eta_1,\eta_2,\ldots,\eta_{M^*}$ are mutually distinct by construction, we deduce that $\text{rank}(D)=2M^*$, $\text{rank}(V_{J-2M})=\text{min}\{J-2M,2M^*\}$ and
    $\text{rank}(V_{2M+1})=\text{min}\{2M+1,2M^*\}$,
where we have used the fact that $V_{J-2M}$ and $V_{2M+1}$ are Vandermonde matrices. From the assumption $J>4M$ and the fact $M\geq M^*$, we obtain 
\ben
\text{rank}(V_{J-2M})=\text{rank}(V_{2M+1})=\text{rank}(D)=2M^*.
\enn
Therefore, there exists $\psi^*\in \C^{(2M+1)\times 1}$ such that $\psi=D V_{2M+1}^T\psi^*$. In terms of the assumption $\phi=V_{J-2M}\psi$, with the help of the factorization \eqref{U-factorization}, we have $\phi=V_{J-2M} D V_{2M+1}^T\psi^*=U\psi^*$, i.e., $\phi\in \mathcal{R}(U)$. 

    \item \textbf{Part two: Define $\Theta(\theta):=(1,\theta,\theta^2,\ldots,\theta^{J-2M-1})^T\in \C^{(J-2M)\times1}$, then we show }
    \begin{equation}\label{NecessarySufficient}
        \Theta(\theta)\in \mathcal{R}(V_{J-2M})\Longleftrightarrow \theta\in \{\xi_1,\xi_2,\ldots,\xi_{M^*},\eta_1,\eta_2,\ldots,\eta_{M^*}\}.
     \end{equation}
    First, let $\theta\in \{\xi_1,\xi_2,\ldots,\xi_{M^*},\eta_1,\eta_2,\ldots,\eta_{M^*}\}$, then $\Theta(\theta)$ is a column of the Vandermonde matrix $V_{J-2M}$ and thus there exists some unit column vector $e_j$ such that $V_{J-2M}e_j=\Theta(\theta)$ for some $j\in\{1,2,\ldots,J\}$, i.e., $\Theta(\theta)\in \mathcal{R}(V_{J-2M})$.
    Let now $\theta \notin \{\xi_1,\xi_2,\ldots,\xi_{M^*},\eta_1,\eta_2,\ldots,\eta_{M^*}\}$. Then
    \begin{equation}
    \text{rank}(V_{J-2M},\Theta(\theta))=\text{min}(J-2M,2M^{*}+1)=2M^{*}+1,
    \nonumber
    \end{equation}
    because  $J>4M$ and $M^*\leq M$.
    Therefore, $\Theta(\theta)\notin \mathcal{R}(V_{J-2M})$ since we have proved in the first part that $\text{rank}(V_{J-2M})=2M^*$.
   
    \item \textbf{Part three: We show the uniqueness of $\{z_1, z_2,\ldots, z_M \}$ from the data matrix $U(x), x\in\G_L$.} \\
    From \eqref{RangeIdentity} and \eqref{NecessarySufficient}, we deduce that $\xi_1,\xi_2,\ldots,\xi_{M^*},\eta_1,\eta_2,\ldots,\eta_{M^*}$ are uniquely determined by the data matrix $U(x)$. Recall from \eqref{xi1-eta<1} that $|\xi_m|=1, |\eta_m|<1$ for all $m\in\{1,2,\ldots,M\}$,  we can conclude that the set $\{\eta_1,\eta_2,\ldots,\eta_{M^*}\}$ is uniquely determined by the data matrix $U(x)$. Furthermore, from the monotonicity of exponential function, we deduce that the set $\mathcal M_x$ is uniquely determined by the data matrix $U(x), x\in\G_L$.
    Finally, the uniqueness of $\{z_1,z_2,\ldots,z_{M^*}\}$ follows by using Lemmas 3.2 and 3.3 in \cite{JiLiu-nearfield} under the assumptions $L\geq6M-2$ and any four sensors in $\Gamma_L$ are non-coplanar.

    \item \textbf{Part four: We complete the proof of this theorem by showing the scattering strengths $\{\tau_1,\tau_2,\ldots,\tau_{M}\}$ can be uniquely determined by the scattered fields $u^s(x,k_j)$, $j=1,2,\ldots,J.$ at the sensor satisfying \eqref{r1neqr2}.} \\
    The existence of the sensor $x\in\G_L$ satisfying \eqref{r1neqr2} has been shown in Lemma 3.3 in \cite{JiLiu-nearfield}. We then show the uniqueness of the scattering strengths. To do so, we define
\ben
     &&\mathbb U
     :=(0,8\pi k_1^2u^s(x,k_1),8\pi k_2^2u^s(x,k_2),\ldots,8\pi k_J^2u^s(x,k_J))^T\in \mathbb C^{(J+1)\times 1},\\
     &&\mathbb T
     :=(\tau_1^*,\tau_2^*,\ldots,\tau_M^*,-\tau_1^*,-\tau_2^*,\ldots,\tau_M^*)^T\in \mathbb C^{2M\times 1},
\enn
where $\tau_m^*={\tau_m}/{r_{m}}$, $m=1,2,\ldots,M$.
Note that 
\begin{equation}
 V_{J+1}=(v_{\xi_1},v_{\xi_2},\ldots,v_{\xi_M},v_{\eta_1},v_{\eta_2},\ldots,v_{\eta_M})\in \mathbb C^{(J+1)\times 2M},
\nonumber
\end{equation}
and using the representation \eqref{R3-scatter} of the scattered field, we have
\begin{equation}
   V_{J+1} \mathbb T=\mathbb U. 
\label{equation}
\end{equation}
Since $\xi_1,\xi_2,\ldots,\xi_M,\eta_1,\eta_2,\ldots,\eta_M$ are distinct by the assumptions \eqref{k0-assumption} and \eqref{r1neqr2}, we have
\begin{equation}
\text{rank}(V_{J+1})=\text{min}(J+1,2M)=2M.
\nonumber
\end{equation}
Therefore, equation \eqref{equation} is uniquely solvable, meaning that $\mathbb T$ can be uniquely obtained by solving equation \eqref{equation}.
\end{itemize}
\end{proof}

\section{Numerical algorithms}
\label{Numerical methods}
\setcounter{equation}{0}
Based on the constructive uniqueness proofs in the previous section, we introduce three numerical algorithms for identifying the point sources. 
We begin with the simplest case for identifying a point source in $\R^3$ from the scattered fields with four sensors and three frequencies. Then we consider the general case with multiple point sources in both $\R^2$ and $\R^3$.

\subsection{Identification of a single point source}
For a single point source with position $z_1\in\R^3$ and scattering strength $\tau_1\in\C$, we take the following scattered fields
\begin{equation}
u^{s}(x, k),\quad x\in \{x_1,x_2,x_3,x_4\},\ k\in\{k_0,2k_0,4k_0\},
\nonumber
\end{equation}
for some $k_0>0$ and four non-coplanar sensors $x_1, x_2, x_3, x_4$. 
Following Theorem \ref{Single point-2}, we first compute its scattering strength using the formula
\begin{equation}
\tau_1=\frac{8\pi k_0^2u^{s}(x_1,k_0)r_{1,1}}{e^{ik_0r_{1,1}}-e^{-k_0r_{1,1}}}
\label{tau-R3-single point}
\end{equation}
where
\ben
r_{1,1}=-\frac{1}{k_0}\text{log}\left[\Re \left(\frac{4u^{s}(x_1,2k_{0})}{u^{s}(x_1,k_{0})}\right)-{\Im \left(\frac{u^{s}(x_1,4k_{0})}{u^{s}(x_1,2k_{0})}\right)}\Bigg/{2\Im\left(\frac{u^{s}(x_1,2k_{0})}{u^{s}(x_1,k_{0})}\right)}\right].
\enn

Note that $r_{1,1}=|x_1-z_1|$, we can similarly obtain the other three distances $r_{j,1}=|x_j-z_1|,\, j=2,3,4$. Then one may compute the position $z_1$ from the one-point determination scheme proposed in \cite{2020Identification}. Alternatively, to determine the position, we introduce a sampling method with the following indicator
\begin{equation}
I_{\R^3-\text{single}}(z)=\left(\sum\limits_{j=1}^{4}\left|\frac{\left|x_{j}-z\right|}{\left|\tau_{1}\right| \left| e^{i k_{0}\left|x_{j}-z\right|}-e^{-k_0\left|x_{j}-z\right|}\right|}-\frac{1}{8 \pi k_{0}^{2} \mid u^{s}\left(x_{j}, k_{0}\right)|}\right|\right)^{-1},\quad z\in \Omega_1,
\label{R^3-Single point}
\end{equation}
where $\Om_1$ is a bounded domain containing $z_1$.
Actually, with the help of the representation of the scattered field given by \eqref{Scattered field}, we observe that the indicator function $I_{\R^3-\text{single}}(z)$ is smooth in $\R^3\ba\{z_1\}$ and blows up at $z=z_1$. 

To sum up, we have the following Algorithm \ref{_R3_single} for identifying a single point source.
\begin{algorithm}[h!]
\label{_R3_single}
\caption{Algorithm for single point source located in $\R^3$}
\LinesNumbered 
1. Collect scattered field patterns $u^s(x,k)$ for all $x\in \{x_1,x_2,x_3,x_4\}$, $k\in\{k_0,2k_0,4k_0\}$;\\
2. Compute the scattering strength $\tau_1$ by the formula \eqref{tau-R3-single point};\\
3. Locate $z_1$ by plotting $I_{\R^3-\text{single}}(z)$, $z\in\Om_1$, given by \eqref{R^3-Single point}.\\
\end{algorithm}
\subsection{Identification of multiple point sources}
We now turn to the much more complex case with multiple point sources. The measurements are the multi-frequency sparse scattered fields
    \begin{equation}
     u^{s}(x, k),\quad x\in \Gamma_L,k\in (k_{-},k_{+}).
     \nonumber
     \end{equation}

\begin{itemize}
    \item \textbf{Identification of multiple point sources in $\R^2$}\\
    According to Theorem \ref{Multipoint sources-11}, we assume that any three sensors in $\Gamma_L$ are not collinear. Our first step is to locate all the point sources by a direct sampling method. Precisely, let $\Omega_2\subset \R^2$ be a bounded domain containing all the point sources. 
   If $\tau_m\in\R, m=1,2,\ldots, M$, following Theorem \ref{Multipoint sources-11}, we define
    \begin{equation}
     I_{\R^2-\text{multiple}}^{\text{real}}(z)=\Bigg|\sum\limits_{x\in \Gamma_L}\int_{k_{-}}^{k_{+}}8 k^{3} \Im(  u^{s}(x, k)) J_{0}(k r) d k\Bigg|,\quad z\in \Omega_2.
     \label{R^2-multipoint-real}
     \end{equation}
    More generally, for $\tau_m\in\C, m=1,2,\ldots, M$, we introduce the following indicator 
     \begin{equation}
     I_{\R^2-\text{multiple}}^{\text{complex}}(z)=\Bigg|\sum\limits_{x\in \Gamma_L}\int_{k_{-}}^{k_{+}}-8i k^{3} u^{s}(x, k) J_{0}(k r) d k\Bigg|,\quad z\in \Omega_2.
     \label{R^2-multipoint}
     \end{equation}
     As shown in the proof of Theorem \ref{Multipoint sources-11}, if the number $L$ is large enough and the frequency band $k_{+}-k_{-}$ is wide enough, both the indicators $I_{\R^2-\text{multiple}}^{\text{real}}(z)$ and$I_{\R^2-\text{multiple}}^{\text{complex}}(z)$ will take local maximum values at $z=z_m, m=1,2,\ldots, M$. 
     Having located all the positions, for each position $z_{m^*}$, we can always find a sensor $x\in \Gamma_L$ such that $r_{m^{*}}\neq r_m,\ \forall m\neq m^*$, then we obtain an approximation of the scattering strength $\tau_{m^{*}}$ by the following formula
     \begin{equation}
     \tau_{m^{*}}=\frac{\int_{k_{-}}^{k_{+}}-8i k^{3}  u^{s}(x, k) J_{0}(k r_{m^*}) d k}{\int_{k_{-}}^{k_{+}} \left(H_{0}^{(1)}\left(k r_{m^*}\right)+\frac{2 i}{\pi} K_{0}\left(k r_{m^*}\right)\right) k J_{0}(kr_{m^*}) d k}.
     \label{R^2-multipoint-tau}
     \end{equation}
     Of course, the resolution of $\tau_{m^{*}}$ depends on the frequency band $k_{+}-k_{-}$. To sum up, we have the following Algorithm \ref{_R2_multipoint} for identifying the point sources in $\R^2$.
\begin{algorithm}[h!]
\label{_R2_multipoint}
\caption{Algorithm for multiple point sources located in $\R^2$}
\LinesNumbered 
1. Collect scattered field patterns $u^s(x,k)$ for all $x\in \Gamma_L$, $k\in(k_{-},k_{+})$. We assume that any three sensors in $\Gamma_L$ are not collinear;\\
2. Locate $z_m, m=1,2,\ldots, M$ by plotting $I_{\R^2-\text{multiple}}^{\text{complex}}(z)$ given \eqref{R^2-multipoint}. In particular, if $\tau_m\in\R, m=1,2,\ldots,M$, one may use the other indicator $I_{\R^2-\text{multiple}}^{\text{real}}(z)$ given by \eqref{R^2-multipoint-real} to locate the positions.\\
3. For each position $z_{m^*}$, compute $\tau_{m^{*}}$ by using the formula \eqref{R^2-multipoint-tau} with a sensor $x\in \Gamma_L$ such that $r_{m^{*}}\neq r_m,\ \forall m\neq m^*$.
\end{algorithm}
\end{itemize}

\begin{itemize}
    \item \textbf{Identification of multiple point sources in $\R^3$}\\
    Similarly,  let $\Omega_3\subset \R^3$ be a bounded domain containing all the point sources. We assume that any four sensors in $\G_L$ are non-coplanar. Following Theorem \ref{Multipoint sources-12}, to locate the point sources, we introduce the following indicator  
     \begin{equation}
     I_{\R^3-\text{multiple}}(z)=\Bigg|\sum\limits_{x\in \Gamma_L}\int_{k_{-}}^{k_{+}}8 k^{2}u^{s}(x,k)e^{-ikr}dk\Bigg|,\quad z\in \Omega_3,
     \label{R^3-multipoint}
     \end{equation}
     After determining all the positions, for each position $z_{m^*}$ we take a proper sensor $x\in \Gamma_L$ in the same way as in $\R^2$ to calculate the corresponding scattering strength by
     \begin{equation}
     \tau_{m^{*}}=\frac{r_{m^*}}{k_{+}-k_{-}}\int_{k_{-}}^{k+}8\pi k^{2}u^{s}\left(x,k\right)e^{-ikr_{m^*}}dk.
     \label{R^3-multipoint-tau}
     \end{equation}
     We summarize the above procedure in the following Algorithm \ref{_R3_multipoint} for identifying the multiple point sources in $\R^3$.
     \begin{algorithm}[h!]
\label{_R3_multipoint}
\caption{Algorithm for multiple point sources located in $\R^3$}
\LinesNumbered 
1. Collect scattered field patterns $u^s(x,k)$ for all $x\in \Gamma_L$, $k\in(k_{-},k_{+})$. We assume that any four sensors in $\G_L$ are not coplanar; \\
2. Locate $\{z_1, z_2,\ldots, z_M\}$ by plotting $I_{\R^3-\text{multiple}}(z)$ given in \eqref{R^3-multipoint}.\\
3. For each position $z_{m^*}$, compute $\tau_{m^{*}}$ by using the formula \eqref{R^3-multipoint-tau} with a sensor $x\in \Gamma_L$ such that $r_{m^{*}}\neq r_m,\ \forall m\neq m^*$.
\end{algorithm}     
\end{itemize}

\section{Numerical examples}
\label{Numerical examples}
\setcounter{equation}{0}
In this section, we present some numerical examples to verify the effectiveness and robustness of the algorithms proposed in  the previous section. 
We perturb the sacttered fields by random noise as follow:
\begin{equation}
u^s(x,k)(1+\text{noise}\times \mathcal{N}),
\nonumber
\end{equation}
where $\mathcal{N}$ is a uniformly distributed pseudorandom number between -1 and 1.
\subsection{A single point source}

We consider a point source in $\R^3$ located at $z_1=(2,2,2)$ with scattering strength $\tau_1=1+i$. By Algorithm \ref{_R3_single}, we use the scattered fields at four sensors $x_1=(1,0,0), x_2=(0,1,0), x_3=(0,0,1), x_4=(-1,-1,-1)$ with three frequencies $k_1=1, k_2=2, k_3=4$. Note that the four sensors are indeed non-coplanar. The sampling domain $\Omega_1$ is $\lbrack1, 3\rbrack^3$ with the sampling space $0.1$. 
Table \ref{table-R3-single} presents a comparison between the true and reconstructed point source under. 
Figure \ref{Picture-R3-Single} shows the indicator $I_{\R^3-\text{single}}(z)$ indeed takes its local maximum at the source position. 
\begin{table}[h!]
\centering
\begin{tabular}{llll}
\hline
& True & Reconstruction with $5\%$ noise & Reconstruction with $10\%$ noise \\
\hline
Location &$(2,2,2)$ &$(2,2,1.9)$ & $(2,2.1,2) $\\
\hline
Strength & $1+i$&$0.9951+0.9457i$ & $1.0188+0.9123i$\\
\hline
\end{tabular}
\caption{Reconstruction of a single point source by the Algorithm \ref{_R3_single}.}
\label{table-R3-single}
\end{table}

\begin{figure}[h!]
\centering
\begin{tabular}{cc}
\subfigure[$L=4$, noise=$5\%$ ]{
\label{Picture-R3-Single1}
\includegraphics[width=.3\textwidth]{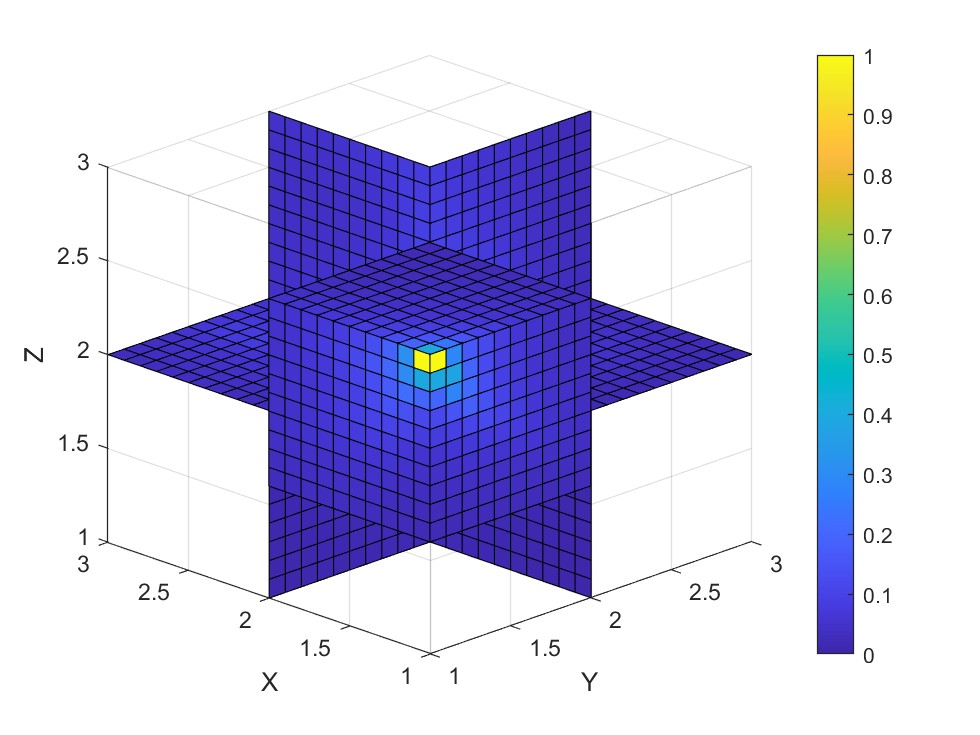}
}\hspace{0em} &
\subfigure[$L=4$, noise=$10\%$ ]{
\label{Picture-R3-Single2}
\includegraphics[width=.3\textwidth]{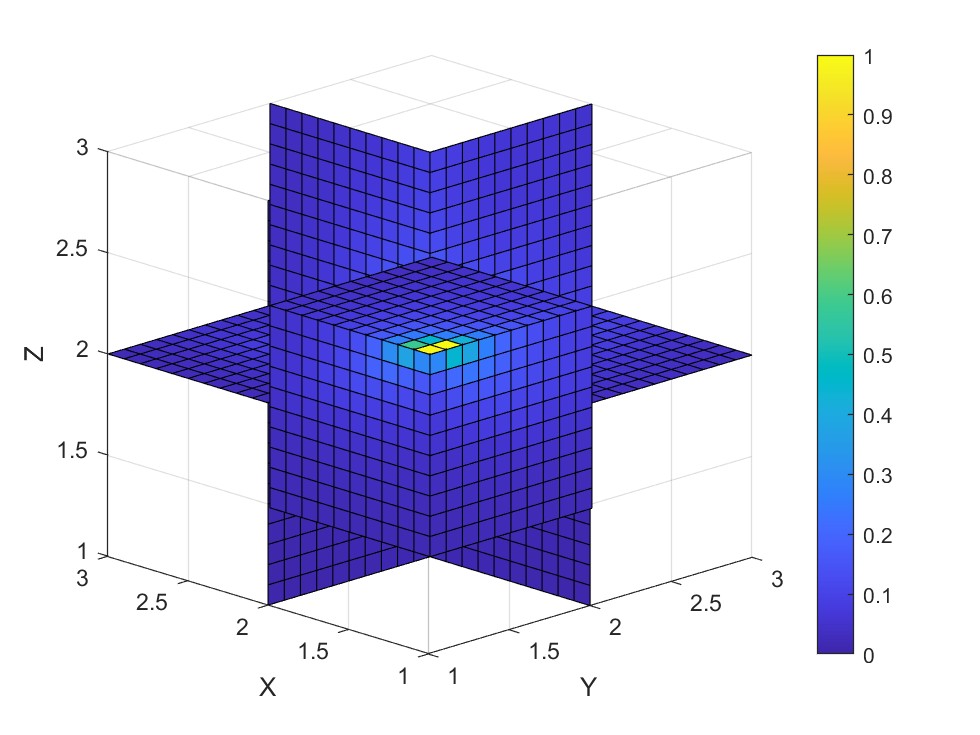}}\\
\end{tabular}
\caption{Locating the position $(2,2,2)$ by $I_{\R^3-\text{single}}(z)$. }
\label{Picture-R3-Single}
\end{figure}
\subsection{Multiple point sources}
Now we turn to the reconstructions of multiple point sources in $\R^2$ or $\R^3$. The scattered fields are taken at equidistant frequencies:
\begin{equation}
k_{0}=k_{-}=1, \quad k_{n}=1+0.1n,\quad n=1,2,\ldots,k_{+}.
\nonumber
\end{equation}
If not otherwise stated, we set $k_{+}=100$ and the measurement noise =$10\%$.

$\bullet$\quad \textbf{Reconstructions of multiple point sources in $\R^2$}

The sampling domain is $\Omega_2=\lbrack0.5,5.5\rbrack^2$ with sampling space $0.01$. 
The scattered fields are taken at $L$ sensors
\begin{equation}  
 x_{j}=(3+5\text{cos}\theta_j,3+5\text{sin}\theta_j),\quad \theta_j={2\pi j}/{L},\ j=0,1,2,\ldots,L-1.
 \label{R2-measurement-points}
 \end{equation}

We begin with a simple example with four point sources located at
\begin{equation}
 z_1=(2, 2),\quad z_2=(2, 4),\quad z_3=(4, 2),\quad{\rm and}\quad z_4=(4, 4).
\nonumber
\end{equation}
The corresponding scattering strengths $\tau_1=1, \tau_2=1.1, \tau_3=1.2$ and $\tau_4=1.3$ are real valued.  Therefore, we take the indicator  $I_{\R^2-\text{multiple}}^{\text{real}}(z)$ for locating the point sources. 

\begin{figure}[h!]
\centering
\begin{tabular}{ccc}
\subfigure[$x_1=(3,0)$]{
\label{TSLZ-R2-4points1}
\includegraphics[width=.27\textwidth]{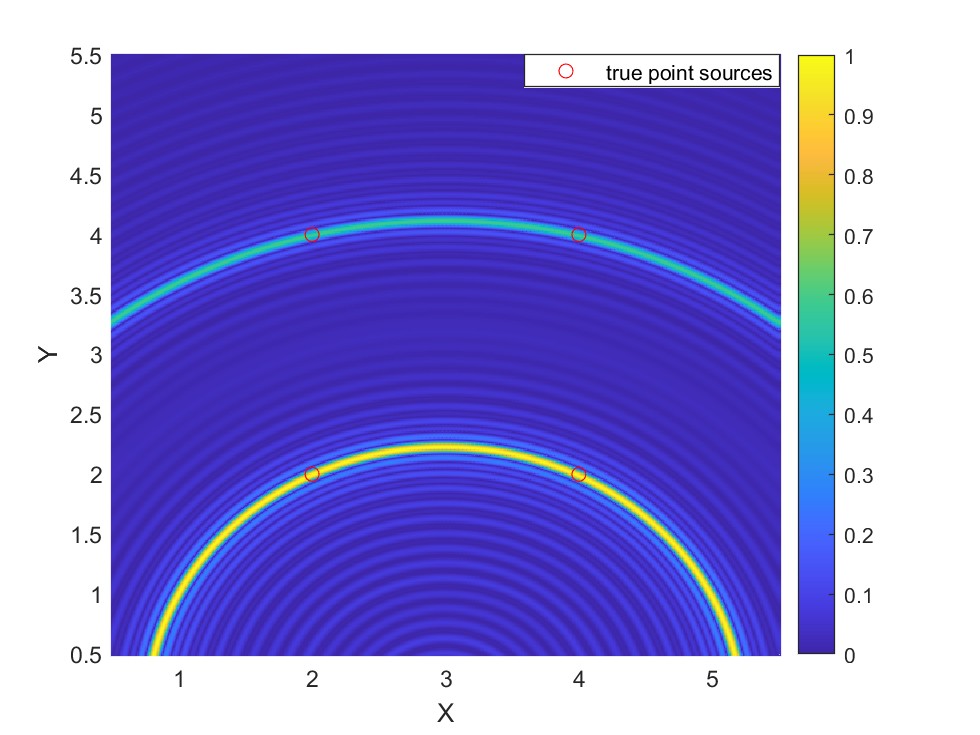}
}\hspace{0em} &
\subfigure[$x_2=(6,6)$]{
\label{TSLZ-R2-4points2}
\includegraphics[width=.27\textwidth]{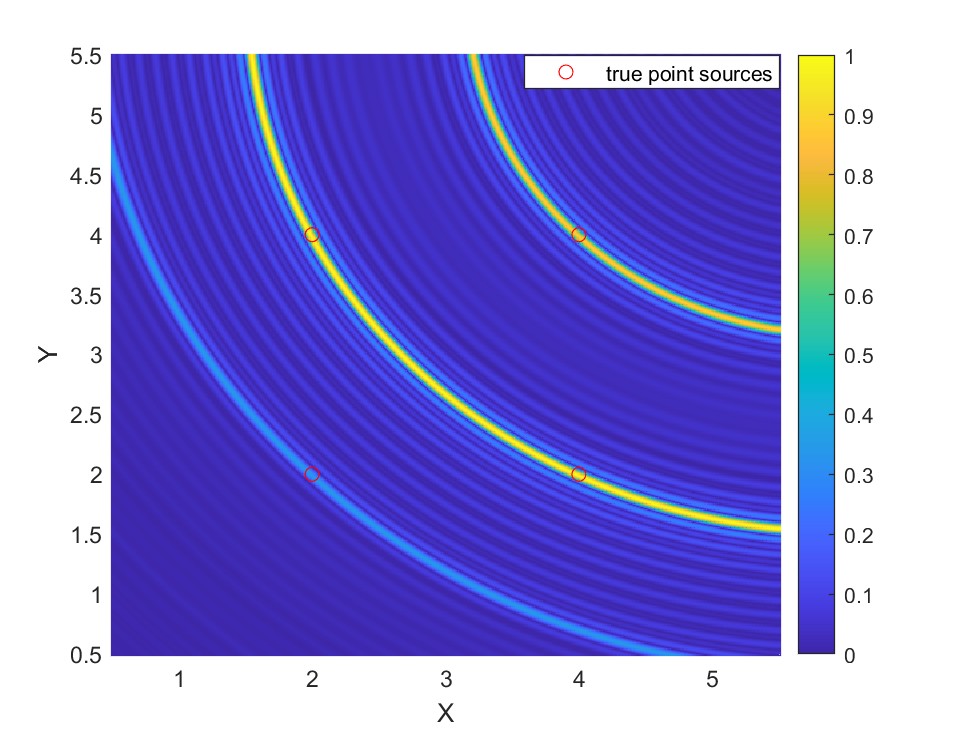}} &
\subfigure[$x_3=(4,0)$]{
\label{TSLZ-R2-4points3}
\includegraphics[width=.27\textwidth]{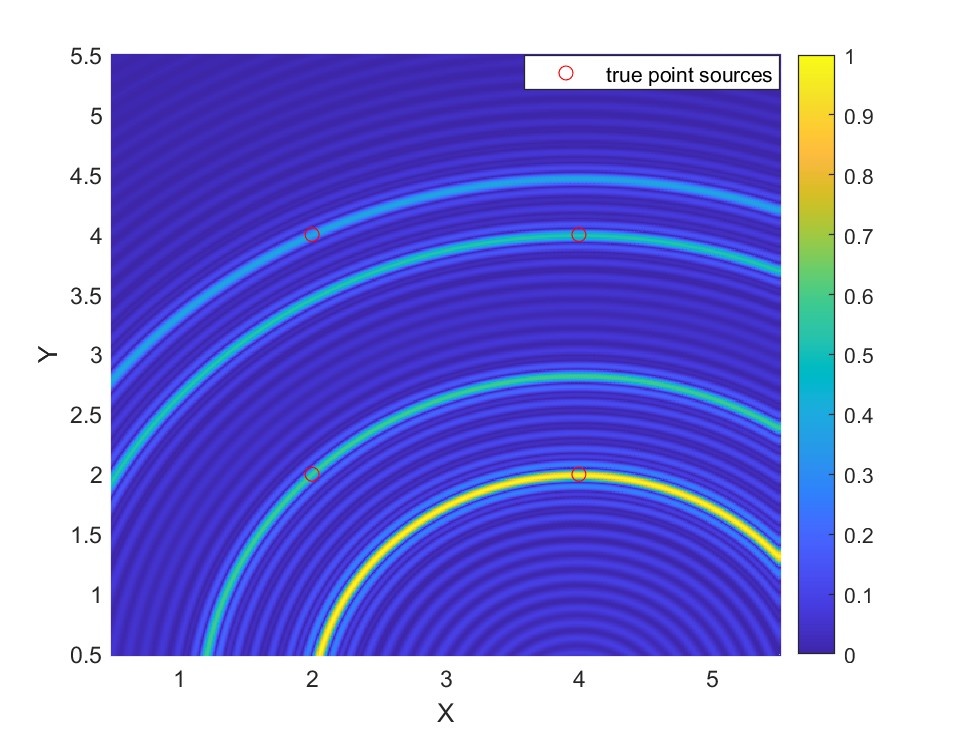}
} \\
\subfigure[$L=4$]{
\label{TSLZ-R2-4points5}
\includegraphics[width=.27\textwidth]{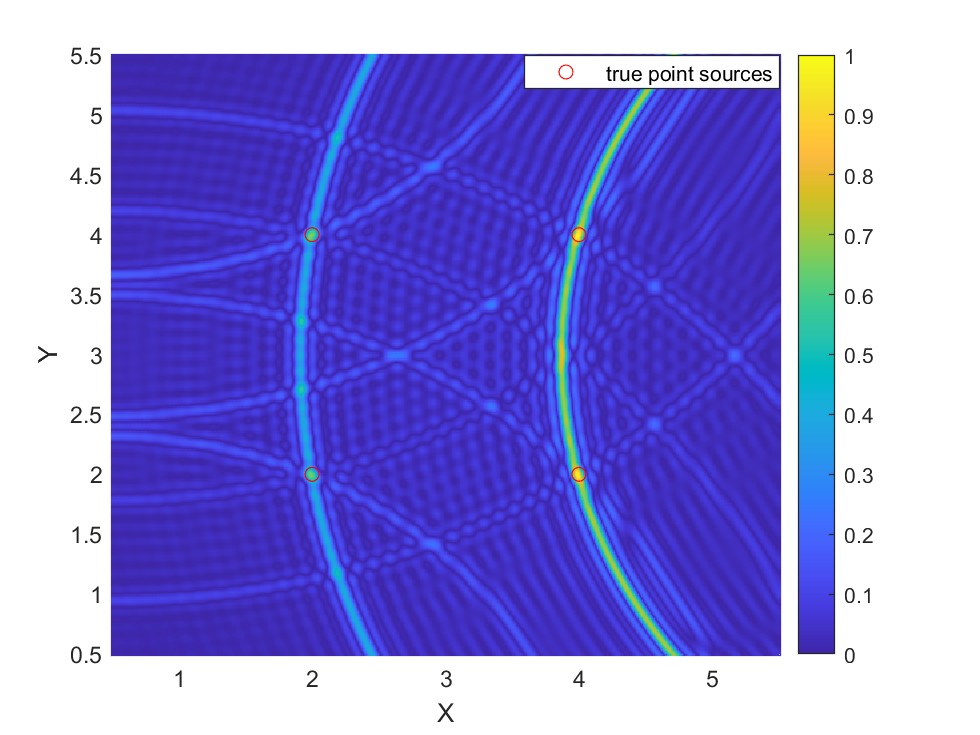}
}\hspace{0em} &
\subfigure[$L=7$]{
\label{TSLZ-R2-4points6}
\includegraphics[width=.27\textwidth]{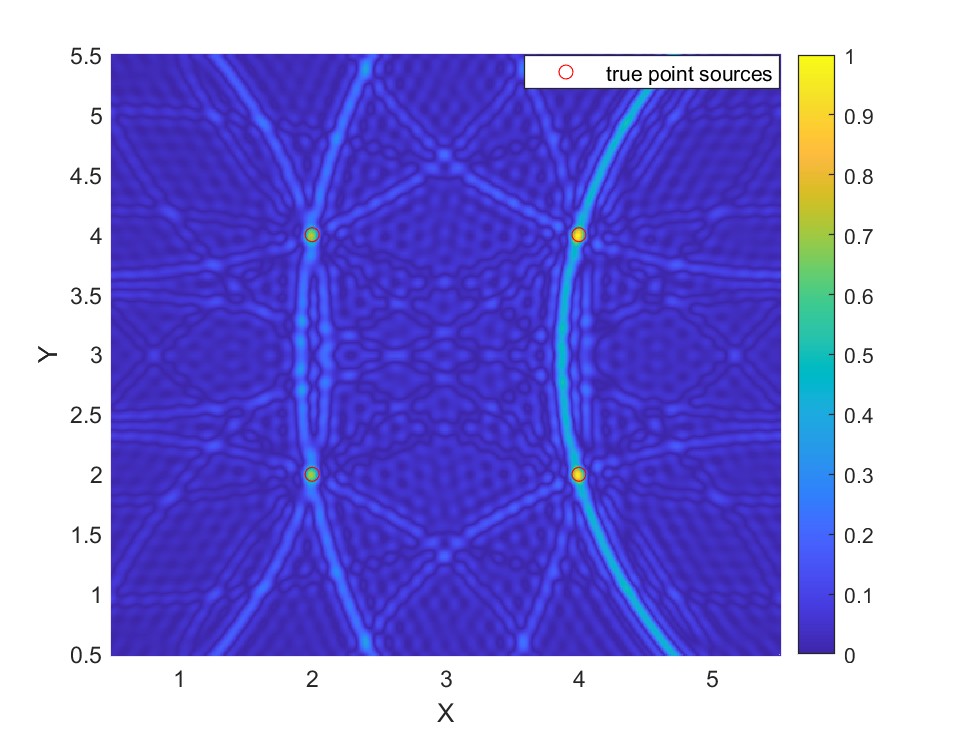}
}&
\subfigure[$L=10$]{
\label{TSLZ-R2-4points4}
\includegraphics[width=.27\textwidth]{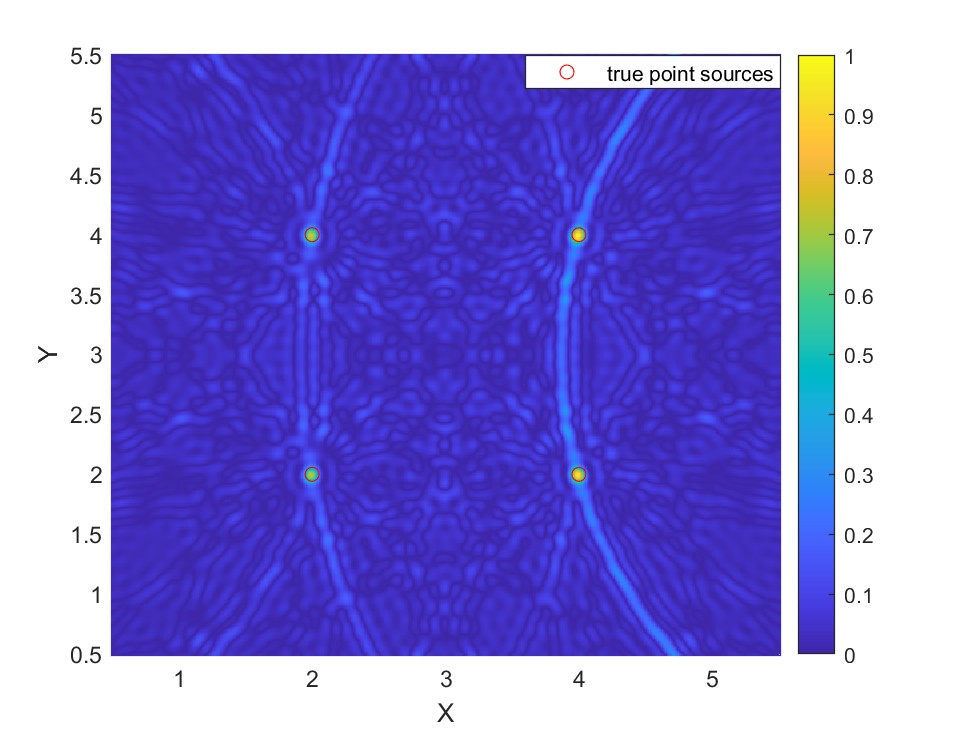}
}
\end{tabular}
\caption{Locating 4 point sources in $\R^2$  by $I_{\R^2-\text{multiple}}^{\text{real}}(z)$ with a few sensors. $k_{+}=50$. }
\label{TSLZ-R2-4points}
\end{figure}

Figure \ref{TSLZ-R2-4points1} shows the reconstruction when the scattered field is only taken at $x_1=(3,0)$. As expected, we observe two circles centered at the sensor $x_1$ passing through all the four point sources. Due to the fact that  $|z_1-x_1|=|z_3-x_1|$ and $|z_2-x_1|=|z_4-x_1|$, only two circles appear in the figure. Similarly, in Figure \ref{TSLZ-R2-4points2} we show the reconstruction with the scattered field taken at $x_2=(6,6)$. We observe three concentric circles with radii $|z_1-x_2|$, $|z_2-x_2|=|z_3-x_2|$ and $|z_4-x_2|$, respectively. 
Figure \ref{TSLZ-R2-4points3} displays four concentric circles by taking the sensor $x_3=(4,0)$. Such a sensor is regarded as a good sensor because the number of circles is the same as the number of the point sources. For this specific example, we expect to avoid using sensors such as $x_1=(3,0)$ and $x_2=(6,6)$. Unfortunately, since the positions of the point sources are what we are looking for, we have no freedom to choose sensors. 
Theorem \ref{Multipoint sources-11} implies that we need $L=4M-1=15$ sensors to uniquely recover $M=4$ point sources. Actually, much fewer sensors are needed in numerical simulations.
As shown in Figure \ref{TSLZ-R2-4points6}, $L=7$ sensors are enough to locate the four point sources. Of course, the reconstruction resolution can be improved by taking more sensors, see, for example, $L=10$ as shown in Figure \ref{TSLZ-R2-4points4}.

 \begin{table}[h!]
    \centering
    \begin{tabular}{llll}
    \hline
      True& Reconstruction& True& Reconstruction \\
       \hline
        $\tau_1=1+1i$&$0.9817 + 0.9598i$&$\tau_7=1.1$&$1.0712 + 0.0244i$\\ $\tau_2=1$&$0.9965 - 0.0904i$&$\tau_8=0.9+0.6i$&$0.8641 + 0.5309i$\\$\tau_3=0.9+1i$&$0.8676 + 0.9111i$&$\tau_9=1$&$1.0349 - 0.0043i$\\ $\tau_4=1.2 $&$1.1081 - 0.0615i$&$\tau_{10}=0.8+0.7i$&$0.7119 + 0.7570i$\\ $\tau_5=0.5+0.9i $&$0.4396 + 0.8906i$&$\tau_{11}=1.3$&$1.2521 + 0.0534i$\\  $\tau_6=0.7+0.8i $&$0.6668 + 0.7943i$\\ 
        \hline
    \end{tabular}
    \caption{Reconstructions of the scattering strengths by  the formula \eqref{R^2-multipoint-tau}.}
    \label{table-R2-F-multipoint}
\end{table}                                

\begin{figure}[h!]
\centering
\begin{tabular}{ccc}
\subfigure[$L=43$]{
\label{Picture-R2-F-multipoint1-complex}
\includegraphics[width=.27\textwidth]{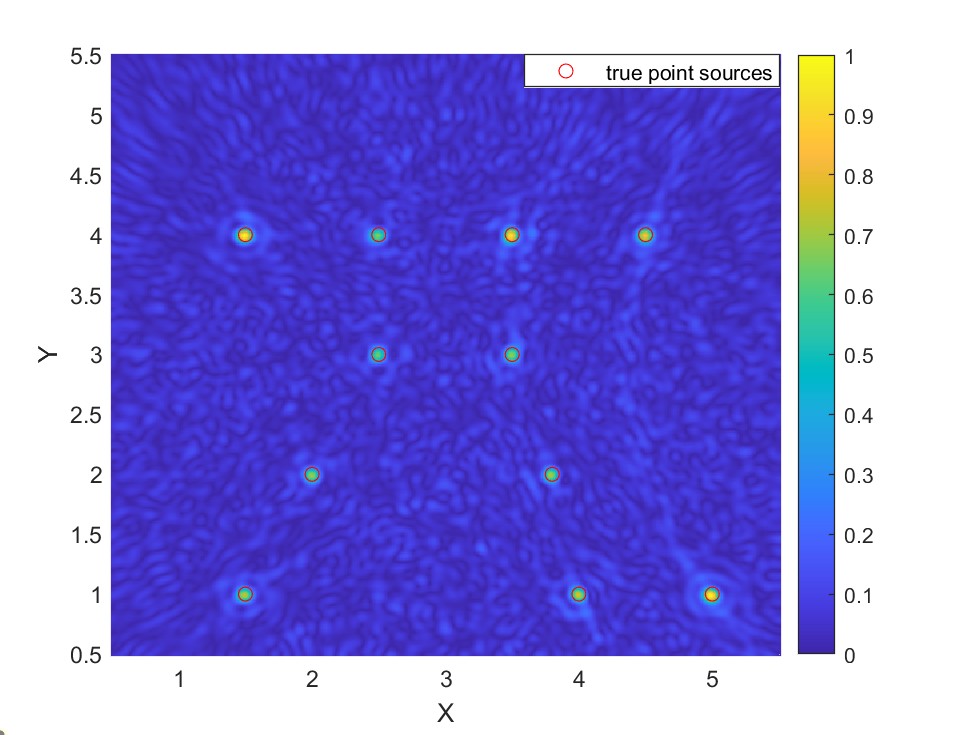}
}\hspace{0em} &
\subfigure[$L=30$]{
\label{Picture-R2-F-multipoint3-complex}
\includegraphics[width=.27\textwidth]{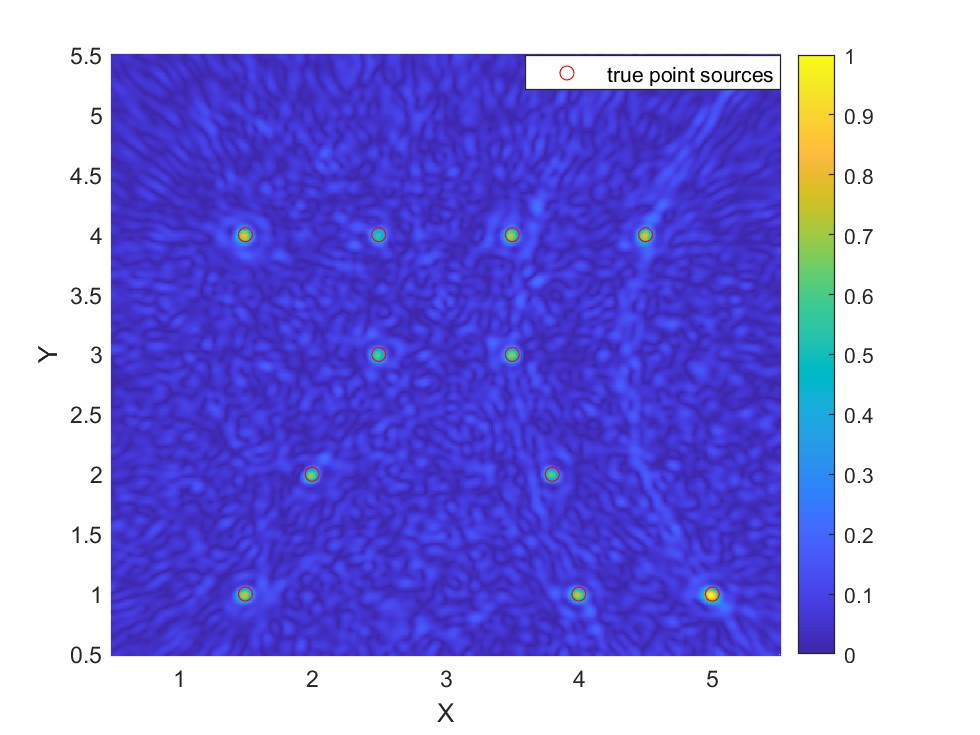}}&
\subfigure[$L=20$]{
\label{Picture-R2-F-multipoint2-complex}
\includegraphics[width=.27\textwidth]{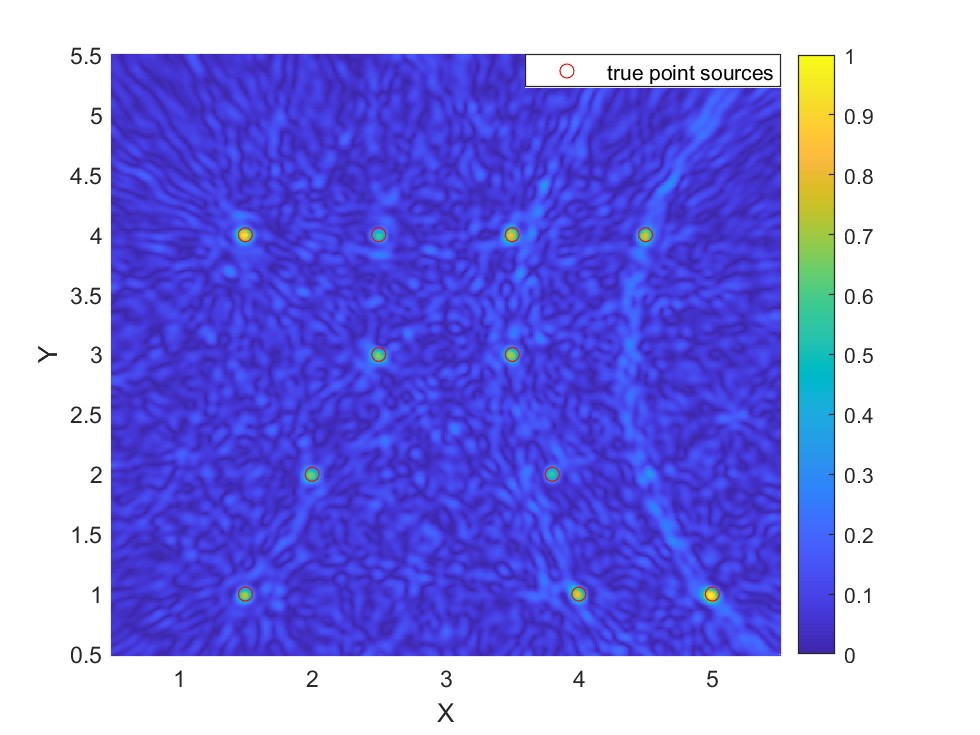}}
\end{tabular}
\caption{Locating 11 $\pi$-shaped point sources in $\R^2$ by $I_{\R^2-\text{multiple}}^{\text{complex}}(z)$. $k_{+}=50$.}
\label{Picture-R2-F-multipoint-complex}
\end{figure}

We then consider a much more complex example with 11 point sources in a $\pi$- shaped formation. The corresponding scattering strengths are listed in Table \ref{table-R2-F-multipoint}.  
 
 Table \ref{table-R2-F-multipoint} gives the comparison of the true strengths with the computed strengths under $10\%$ noise, and the strengths are well-reconstructed. Figure \ref{Picture-R2-F-multipoint-complex} shows the location reconstructions by plotting $I_{\R^2-\text{multiple point}}^{\text{complex}}(z)$ with $L=43$, $L=30$ and $L=22$, respectively. The number $L=43=4M-1$ is suggested by Theorem \ref{Multipoint sources-11}. We observe again that, with less sensors ($L=20$), all the 11 point sources are well captured even when $10\%$ relative noise is considered. Of course, the resolution can be improved by using more sensors (e.g., $L=30$ and $L=43$).\\

$\bullet$\quad \textbf{Reconstruction of multiple point sources in $\R^3$}

Finally, we present an example with four point sources in $\R^3$. The true positions are
\begin{equation}
 z_1=(1,0,0),\, z_2=(0,2,0),\, z_3=(2,1,0)\, {\rm and }\,z_4=(0,0,1.5).
\nonumber
\end{equation}
The sampling domain $\Omega_3=\lbrack0.5,2.5\rbrack^3$ with sampling space $0.1$. The measurement sensors are
 \begin{equation}  
 x_{j}=(1+3\text{sin}\theta_j\text{cos}\phi_j,1+3\text{sin}\theta_j\text{sin}\phi_j,1+3\text{cos}\theta_j), 
 \nonumber
 \end{equation}
 with $\theta_j:=\arccos{\left({2j}/{L}-1\right)}$, $\phi_j={2\pi j}/{L}$, $j=0,1,2,\ldots,L-1$.

\begin{figure}[h!]
\centering
\begin{tabular}{cc}
\subfigure[$L=22$]{
\label{Picture-R3-multipoint1}
\includegraphics[width=.3\textwidth]{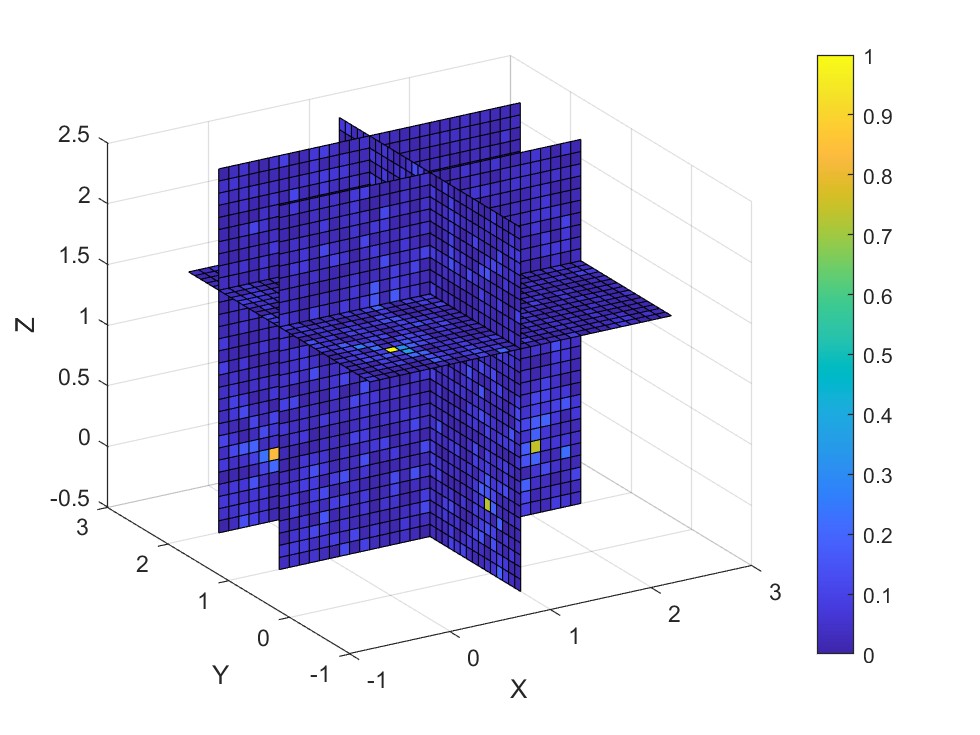}
}\hspace{0em}
\subfigure[$L=11$]{
\label{Picture-R3-multipoint2}
\includegraphics[width=.3\textwidth]{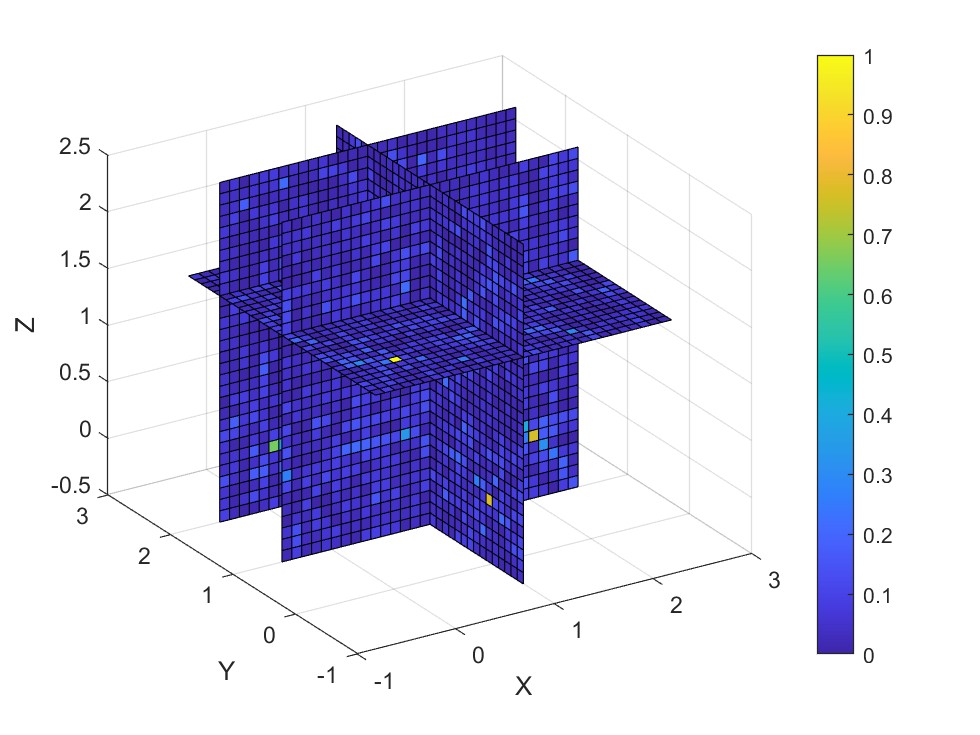}}
\label{Picture-R3-multipoint3}
\subfigure[$L=11$]{
\includegraphics[width=.3\textwidth]{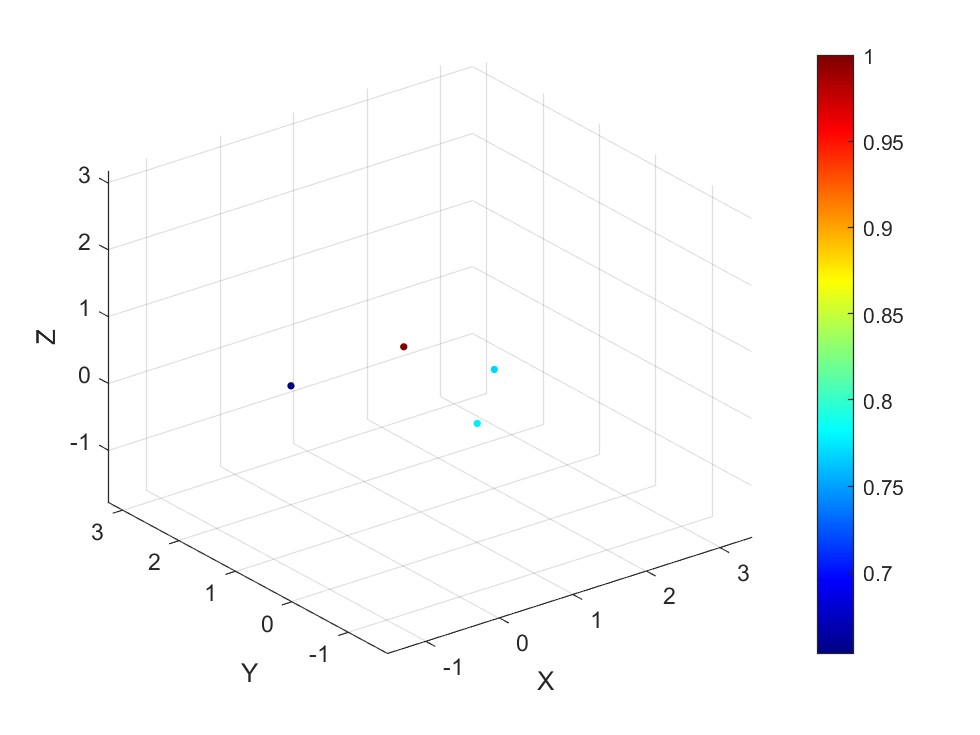}
}\hspace{0em} 
\end{tabular}
\caption{Locating 4 point sources in $\R^3$ by $ I_{\R^3-\text{multiple}}(z)$. }
\label{Picture-R3-multipoint}
\end{figure}

Figure \ref{Picture-R3-multipoint} shows the location reconstructions by $ I_{\R^3-\text{multiple}}(z)$. Obviously, the scattered fields at $11$ sensors are enough to locate all the four point sources. Table \ref{table-R3-multipoint} gives a comparison between the true strengths and the computed strengths. Considering the $10\%$ relative noise, the reconstructions are quite satisfactory.

\begin{table}[h!]
    \centering
    \begin{tabular}{lllll}
    \hline
     True& $\tau_1=1+1i$ &$\tau_2=1-1i$&$\tau_3=1+1.5i$&$\tau_4=1.5+1i$\\
       \hline
         Reconstruction& $0.9972+ 1.0337i $ &$ 0.9996 - 1.0033i $& $0.9804 + 1.5136i$& $1.5398 + 0.9696i $\\ 
        \hline
    \end{tabular}
    \caption{Reconstructions of the scattering strengths by  the formula \eqref{R^3-multipoint-tau}.}
    \label{table-R3-multipoint}
\end{table}

\section*{Acknowledgment}
 The research of X. Liu is supported by the National Key R\&D Program of China under grant 2024YFA1012303 and the NNSF of China under grant 12371430.

\bibliographystyle{SIAM}

\end{document}